\documentclass{amsart}

\usepackage{graphicx}
\usepackage{caption}
\usepackage{subcaption}
\usepackage{amsmath,amssymb,amsfonts}
\usepackage{placeins}
\usepackage{lscape}

\newtheorem{thm}{Theorem}[section]
\newtheorem{prop}[thm]{Proposition}
\newtheorem{lem}[thm]{Lemma}

\newtheorem{conj}[thm]{Conjecture}

\begin{document}

\title{Integral finite surgeries on knots in $S^3$}

\author{Liling GU}
\address{Department of Mathematics, Caltech,
1200 E California Blvd, Pasadena, CA 91125}
\email{gll007$@$caltech.edu}
\maketitle

\begin{abstract}
Using the correction terms in Heegaard Floer homology, we prove that if a knot in $S^3$ admits a positive integral $\mathbf{T}$-,  $\mathbf{O}$- or  $\mathbf{I}$-type surgery,  it must have the same knot Floer homology as one of the knots given in our complete list, and the resulting manifold is orientation-preservingly homeomorphic to the $p$-surgery on the corresponding knot.
\end{abstract}

\section{Introduction}
In the early 1960s, Wallace \cite{Wallace} and Lickorish \cite{Lickorish} proved independently that any closed, orientable, connected 3-manifold can be obtained by performing Dehn surgery on a framed link in the 3-sphere. One natural question is which manifolds can be obtained by some surgery on a knot. In this paper we consider the manifolds with finite noncyclic fundamental groups. By Perelman's resolution of the Geometrization Conjecture\cite{P1, P2, P3}, the manifolds with finite fundamental group are spherical space forms. They fall into five classes, those with
cyclic $\pi_1$ and those with finite $\pi_1$ based on the four isometries of a sphere:

\begin{thm} (Seifert \cite{Seifert}). If $Y^3$ is closed, oriented and Seifert-fibered
with finite but noncyclic fundamental group, then it has base orbifold $S^2$ and
is one of:
\begin{enumerate}
\item
Type $\mathbf{D}$, dihedral:
$(b;\frac{1}{2},\frac{1}{2},\frac{a_3}{b_3})$
\item
Type $\mathbf{T}$, tetrahedral:
$(b;\frac{1}{2},\frac{a_2}{3},\frac{a_3}{3})$
\item
Type $\mathbf{O}$, octahedral:
$(b;\frac{1}{2},\frac{a_2}{3},\frac{a_3}{4})$
\item
Type $\mathbf{I}$, icosahedral:
$(b;\frac{1}{2},\frac{a_2}{3},\frac{a_3}{5})$
\end{enumerate}
with coefficients chosen so that $(a_i, b_i) =1$.
\end{thm}

Berge \cite{Berge} constructed a list of knots which yield lens space surgeries, and he conjectured that it is complete. Greene \cite{Greene} proved that if a $p$-surgery along a knot produces a lens space $L(p, q)$, there exists a $p$-surgery along a Berge knot with the same knot Floer homology groups as $L(p,q)$. In this paper we focus on knots with finite noncyclic surgeries. According to Thurston \cite{Thurston}, a knot is either a torus knot, a hyperbolic knot, or a satellite knot. Moser \cite{Moser} classified all finite surgeries on torus knots, and Bleiler and Hodgson \cite{Bleiler_Hodgson} classified all finite surgeries on cables of torus knots. Boyer and Zhang \cite{Boyer_Zhang} showed that a satellite knot with a finite noncyclic surgery must be a cable of some torus knot. There is also some progress about hyperbolic knots. Doig\cite{Doig} proved that there are only finitely many spherical space forms which come from a $p/q$-surgery on $S^3$ for a fixed integer $p$. From now on we will only consider Dehn surgeries on hyperbolic knots. 

Suppose $M$ is a 3-manifold with torus boundary, $\alpha$ is a slope on $\partial M$.
Let $M(\alpha)$ be the Dehn filling along $\alpha$. If $M$ is hyperbolic, Thurston's Hyperbolic Dehn Surgery Theorem says that at most finitely many of the fillings are nonhyperbolic. These surgeries are called $exceptional$ $surgeries$.  In \cite{Boyer_Zhang}, Boyer and Zhang showed that if M is hyperbolic, $M(\alpha)$ has a finite
fundamental group and $M(\beta)$ has a cyclic fundamental group, then $|\Delta(\alpha, \beta)| \leq 2$. In particular, if the $p/q$-surgery on a hyperbolic knot $K \subset S^3$, denoted $S^3_K(p/q )$, has a finite fundamental
group, then $|q| \leq 2$. For $|q|=2$, Li and Ni \cite{Li} prove that $K$ has the same
knot Floer homology as either $T(5, 2)$ or a cable of a torus knot (which must be
$T(3, 2)$ or $T(5, 2)$). From now on we will only consider integral surgeries on hyperbolic knots.

Taking the mirror image of a knot $K$ if necessary, we may assume $p>0$. We consider here all $\mathbf{T}$-,  $\mathbf{O}$- and  $\mathbf{I}$-type spherical space forms. In this paper, all manifolds are oriented. If $Y$ is an oriented manifold, then $-Y$ denotes the same manifold with the opposite orientation.
Let $\mathbb{T}$ be the exterior of the right-hand trefoil, then $\mathbb{T}(p/q)$ is the manifold obtained by $p/q$-surgery on the right hand trefoil. It is well-known that any $\mathbf{T}$- $\mathbf{O}$-, or
$\mathbf{I}$-type manifold is homeomorphic to a $\pm \mathbb{T}(p/q)$ (see Lemma \ref{Lemma}).

Our main result is the following theorem.
\begin{thm}
Suppose that $K$ is a knot in $S^3$, and that the $p$-surgery
on K is a $\mathbf{T}$-,  $\mathbf{O}$- or  $\mathbf{I}$-type spherical space form for some integer $p > 0$, then K has the same
knot Floer homology as one of the knots $\widetilde K$ in table \ref{table: torus knots and satellite knots} or table \ref{table: hyperbolic knots}, and the resulting manifold is orientation-preservingly homeomorphic to the $p$-surgery on the corresponding knot $\widetilde K$.
\label{thm}
\end{thm}

The strategy of our proof is to compute the Heegaard Floer correction terms
for the $\mathbf{T}$-,  $\mathbf{O}$- and $\mathbf{I}$-type manifolds, and to compare them with the correction terms
of the integral surgeries on knots in $S^3$. If they match, the knot Floer
homology of the knots can be recovered from the correction terms.
Knot Floer homology contains a lot of information about knots. For example,
it detects the genus  \cite{Genus} and fiberedness \cite{Ni}. We propose the following conjecture.

\begin{conj}
Suppose that $K$ is a hyperbolic knot in $S^3$, and that the $p$-surgery
on K is a $\mathbf{T}$-,  $\mathbf{O}$- or  $\mathbf{I}$-type spherical space form for some integer $p > 0$, then K is one of the knots in table \ref{table: hyperbolic knots}.
\end{conj}

\subsection*{Acknowledgments} I wish to thank my advisor, 
Yi Ni, for his patience and instructions during this work. I would also like to thank John Berge for kindly sending me his papers, and Xingru Zhang for pointing out that I may use Dean's twist torus knots.

\section{Preliminaries on Heegaard Floer homology and
correction terms\label{sec}}

Heegaard Floer homology was introduced by Ozsv$\acute{\text{a}}$th and Szab$\acute{\text{o}}$ \cite{HeegaardFloer}. Given
a closed oriented 3-manifold Y and a Spin$^c$ structure $\mathfrak{s} \in$ Spin$^c(Y )$, one can
define the Heegaard Floer homology $\widehat{HF}(Y, \mathfrak{s}),HF^+(Y, \mathfrak{s} )$, . . . , which are invariants of $(Y, \mathfrak{s} )$. When $\mathfrak{s} $ is torsion, there is an absolute $\mathbb{Q}$-grading on $HF^+(Y, \mathfrak{s} )$.
When $Y$ is a rational homology sphere,  Ozsv$\acute{\text{a}}$th and Szab$\acute{\text{o}}$ \cite{FourManifold} defined a $correction$
$term$ $d(Y, \mathfrak{s}) \in \mathbb{Q}$, which is the shifting of the absolute grading of $HF^+(Y, \mathfrak{s})$ relative to $HF^+(S^3, 0)$.

The correction terms have the following symmetries:
\begin{equation}
d(Y, \mathfrak{s}) = d(Y, J\mathfrak{s}), \quad d(-Y, \mathfrak{s}) = -d(Y, \mathfrak{s}),
\label{conjugation}
\end{equation}
where $J$ : Spin$^c(Y )$ $\to$ Spin$^c(Y )$ is the conjugation.

Suppose that $Y$ is an integral homology sphere, $K \subset Y$ is a knot. Let
$Y_K(p/q)$ be the manifold obtained by $p/q$-surgery on $K$. Ozsv$\acute{\text{a}}$th and Szab$\acute{\text{o}}$
defined a natural identification $\sigma: \mathbb{Z}/p \mathbb{Z} \to$ Spin$^c(Y_K(p/q))$ \cite{FourManifold, Rational}. For simplicity,
we often use an integer $i$ to denote the Spin$^c$ structure $\sigma([i])$, when
$[i] \in \mathbb{Z}/p \mathbb{Z}$ is the congruence class of $i$ modulo $p$.

A rational homology sphere $Y$ is an $L$-$space$ if rank$\widehat{HF}(Y ) = |H_1(Y )|$. Examples of L-spaces include spherical space forms. The information about the
Heegaard Floer homology of an L-space is completely encoded in its correction
terms.

Let $L(p, q)$ be the lens space obtained by $p/q$-surgery on the unknot. The
correction terms for lens spaces can be computed inductively as follows:
$$
d(S^3,0)=0,
$$
\begin{equation}
d(-L(p,q), i)=\frac{1}{4}-\frac{(2i+1-p-q)^2}{4pq}-d(-L(q,r),j)
\label{recursion}
\end{equation}
where $0 \leq i < p+q$, $r$ and $j$ are the reductions of $p$ and $i$ modulo $q$, respectively.

For example, using the recursion formula (\ref{recursion}), we get
\begin{align}
&d(L(3,q),i)= \left\{ 
  \begin{array}{l l}
    (\frac{1}{2}, -\frac{1}{6}, -\frac{1}{6}) & \quad \text{$q=1, i = 0, 1, 2$}\\
    (\frac{1}{6}, \frac{1}{6}, -\frac{1}{2}) & \quad \text{$q=2, i = 0, 1, 2$}\\
  \end{array} \right.\\
&d(L(4,q),i)= \left\{ 
  \begin{array}{l l}
    (\frac{3}{4}, 0, -\frac{1}{4}, 0) & \quad \text{$q=1, i = 0, 1, 2, 3$}\\
    (0, \frac{1}{4}, 0, -\frac{3}{4}) & \quad \text{$q=3, i = 0, 1, 2, 3$}\\
  \end{array} \right.\\
&d(L(5,q),i)= \left\{ 
  \begin{array}{l l}
    (1, \frac{1}{5}, -\frac{1}{5}, -\frac{1}{5}, \frac{1}{5}) & \quad \text{$q=1, i = 0, 1, 2,3,4$}\\
    (\frac{2}{5}, \frac{2}{5}, -\frac{2}{5}, 0, -\frac{2}{5}) & \quad \text{$q=2, i = 0, 1, 2,3,4$}\\
   (\frac{2}{5}, 0, \frac{2}{5}, -\frac{2}{5}, -\frac{2}{5}) & \quad \text{$q=3, i = 0, 1, 2,3,4$}\\
   (-\frac{1}{5}, \frac{1}{5}, \frac{1}{5}, -\frac{1}{5},-1) & \quad \text{$q=4, i = 0, 1, 2,3,4$}\\
  \end{array} \right.
\label{345}
\end{align}

Given a null-homologous knot $K \subset Y$ , Ozsv$\acute{\text{a}}$th and Szab$\acute{\text{o}}$ \cite{KnotFloer} and Rasmussen
\cite{Rasmussen} defined the knot Floer homology. From \cite{Rational}, if we know the knot Floer homology, then we can compute the
Heegaard Floer homology of all the surgeries on K. In particular, if the $p/q$-surgery on $K \subset S^3$ is an L-space surgery, where $p, q > 0$, then the correction
terms of $S^3_K(p/q)$ can be computed from the Alexander polynomial of $K$ as follows.

Suppose
$$
\Delta_K(T)=a_0+\sum_{i>0} a_i (T^i+T^{-i})
$$
Define a sequence of integers
$$
t_i =\sum_{j=1}^\infty ja_{i+j} ,\quad  i \geq 0.
$$
then $a_i$ can be recovered from $t_i$ by
\begin{equation}
a_i = t_{i-1} - 2t_i + t_{i+1}, \quad \text{for} \  i > 0.
\end{equation}
If $K$ admits an L-space surgery, then one can prove\cite{Rational, Rasmussen}
\begin{equation}
t_s \geq 0,  \quad t_s \geq t_{s+1} \geq t_s-1, \quad t_{g(K)}=0.
\label{ts_nonincreasing}
\end{equation}
Moreover, the following proposition holds.

\begin{prop}
Suppose the $p/q$-surgery on $K\subset S^3$ is an L-space surgery, where $p,q>0$. Then for any $0\leq i < p$ we have
$$
d(S^3_K(p/q),i)=d(L(p,q),i)-2\max\{t_{\lfloor\frac{i}{q}\rfloor}, t_{\lfloor\frac{p+q-1-i}{q}\rfloor}\}.
$$
\label{prop}
\end{prop}
This formula is contained in  Ozsv$\acute{\text{a}}$th and Szab$\acute{\text{o}}$ \cite{Rational} and Rasmussen \cite{Rasmussen}. 

\section{The strategy of our proof}
Recall that $\phi: \mathbb{Z}/p \mathbb{Z} \to$ Spin$^c(Y_K(p/q))$ is the natural identification defined by Ozsv$\acute{\text{a}}$th and Szab$\acute{\text{o}}$.
The following two lemmas are from \cite{Li}.

\begin{lem}
Suppose i is an integer satisfying $0 \leq i < p + q$, then $J(\sigma([i]))$ is
represented by $p + q - 1 - i$.
\label{lemma_J}
\end{lem}

\begin{lem}
Any $\mathbf{T}$-type manifold is homeomorphic to $\pm \mathbb{T}( \frac{6q\pm3}{q} )$ for some positive integer $q$ with $(q, 3) = 1$.  Any  $\mathbf{O}$-type manifold is homeomorphic to
$\pm\mathbb{T}( \frac{6q\pm4}{q} )$ for some positive integer $q$ with $(q, 2) = 1$. Any  $\mathbf{I}$-type manifold is homeomorphic to
$\pm\mathbb{T}( \frac{6q\pm5}{q} )$ for some positive integer $q$ with $(q, 5) = 1$.
\label{Lemma}
\end{lem}

Let $p, q>0$ be coprime integers. Using Proposition \ref{prop}, we get

\begin{equation}
 d( \mathbb{T}(p/q),i) = d(L(p,q), i) - 2 \chi_{[0,q)}(i),
\label{T32complement}
\end{equation}

where $ \chi_{[0,q)}(i) = \left\{ 
  \begin{array}{l l}
    1 & \quad \text{when  $0 \leq i<q$}\\
    0 & \quad \text{when  $q\leq i <p$}\\
  \end{array} \right.$  

Suppose $S^3_K(p)$ is a spherical space form, then by Proposition \ref{prop},

\begin{align*}
d(S^3_K(p),i)&=d(L(p,1), i)-2\max\{t_i, t_{p-i}\}\\
&=\frac{(2i-p)^2-p}{4p}-2 t_{\min(i, p-i)}
\end{align*}

If $S^3_K(p)\cong \varepsilon\mathbb{T}(p/q)$, $\varepsilon = \pm 1$, then the two sets
$$
\{ d(S^3_K(p),i)| i \in \mathbb{Z}/p\mathbb{Z}\}, \quad \{ \varepsilon d(\mathbb{T}(p/q),i)| i \in \mathbb{Z}/p\mathbb{Z}\}
$$
are equal. However, the two parametrizations of Spin$^c$ structures may differ by an affine isomorphism of $\mathbb{Z}/p\mathbb{Z}$. More precisely, there exists an affine isomorphism $\phi:  \mathbb{Z}/p\mathbb{Z} \to \mathbb{Z}/p\mathbb{Z}$, such that

$$
d(S^3_K(p),i)=\varepsilon d(\mathbb{T}(p/q),\phi(i)).
$$

 For any integers $a, b$, define $\phi_{a,b}:  \mathbb{Z}/p\mathbb{Z} \to \mathbb{Z}/p\mathbb{Z}$ by

$$\phi_{a,b}(i)=ai+b \mod p$$.

\begin{lem}
There are at most two values for $b$, $b_j=\frac{jp+q-1}{2}$, $j = 0, 1$.
\label{b}
\end{lem}
\begin{proof}
The affine isomorphism $\phi$ commutes with $J$, i.e., $\phi J_{p} = J_{\frac{p}{q}} \phi $. Using lemma \ref{lemma_J}, we get the desired values for $b$. Note that $b_0$ or $b_1$ may be a half-interger, in this case we discard it.
\end{proof}

Note $\phi_{a,b}(i) =\phi_{p-a,b}(p-i)$, By (\ref{conjugation}) and Lemma \ref{lemma_J}, $$d(\mathbb{T}(p/q), \phi_{a,b}(i)) = d(\mathbb{T}(p/q), \phi_{p-a,b}(p-i))= d(\mathbb{T}(p/q), \phi_{p-a,b}(i)).$$ 

So we may assume
\begin{equation}
0<a<\frac{p}{2}, \quad (p,a)=1.
\label{a}
\end{equation}

Then we may assume
$$
d(S^3_K(p),i)= \varepsilon d(\mathbb{T}(p/q),\phi_{a,b_j}(i)), \quad \text{for some $a$, any $i \in  \mathbb{Z}/p\mathbb{Z}$, and $j=0$ or 1.}
$$

Let
\begin{equation}
\Delta_{a,b_j}^\varepsilon(i)=d(L(p,1), i)-\varepsilon d(\mathbb{T}(p/q),\phi_{a,b_j}(i))
\label{delta_plus}
\end{equation}

By Proposition \ref{prop}, we should have
\begin{equation}
\Delta_{a,b_j}^\varepsilon(i)=2 t_{\min(i, p-i)}
\label{Delta}
\end{equation}
if $S^3_K(p)\cong \varepsilon \mathbb{T}(p/q)$ and $\phi_{a,b_j}$ identifies their Spin$^c$ structures.

In order to prove Theorem \ref{thm}, we will compute the correction terms of the
$\mathbf{T}$-,$\mathbf{O}$- and $\mathbf{I}$-type manifolds using (\ref{T32complement}). For all $a$ satisfying (\ref{a}), we compute the sequences $\Delta_{a,b_j}^\varepsilon(i)$. Then we check whether they satisfy (\ref{Delta})
for some $\{t_s\}$ as in (\ref{ts_nonincreasing}). We will show that (\ref{Delta}) cannot be satisfied when
$p$ is sufficiently large. For small $p$, a direct computation yields all the $p/q$'s. By a
standard argument in Heegaard Floer homology \cite{Lens}, we can get the knot Floer
homology of the corresponding knots, which should be the knot Floer homology
of either a $(p,q)$-torus knot ($(p,q)=(2,3), (2,5), (3,4), (3,5)$), a cable knot or some hyperbolic knot. We will list torus knots and cables of torus knots seperately for completeness, one may also consult Moser \cite{Moser} and Bleiler and Hodgson \cite{Bleiler_Hodgson}.

\section{The case when $p$ is large}

In this section, we will assume that $S^3_K(p)\cong
\varepsilon \mathbb{T}(p/q)$, and 

$$p = 6q + \zeta r, r\in \{3,4,5\}, \varepsilon, \zeta \in \{-1,1\}.$$

We will prove that this does not happen when p is sufficiently large.

\begin{prop}
If $p >310r(36r+1)^2$, then $S^3_K(p)\not\cong \varepsilon \mathbb{T}(p/q)$, where $p = 6q + \zeta r, r\in \{3,4,5\}$.
\end{prop}
Let $s\in \{0,1,...,r-1\}$ be the reduction of $q$ modulo $r$. For any integer $n$, let $\theta(n)\in \{0,1\}$ be the reduction of $n$ modulo $2$, and let $\bar{\theta}(n) = 1- \theta(n)$.

\begin{lem}
For $0\leq i<q$,
\begin{equation}
d(L(q, \frac{1-\zeta}{2}q+\zeta r), i)=\zeta \left(\frac{(2i+1-q-\zeta r)^2}{4qr}-\frac{1}{4}-d(L(r,s), i \mod r)\right).
\label{main_case}
\end{equation}
\end{lem}

\begin{proof}
For $0\leq i<q$, using (\ref{recursion}), we have 

\begin{align*}
d(L(q, r), i)&=\frac{(2i+1-q-r)^2}{4qr}-\frac{1}{4}-d(r,s, i\mod r)\\
d(L(q, q-r), i)
&=\frac{(2i+1-2q+r)^2}{4q(q-r)}-\frac{1}{4}-d(L(q-r, r), i)\\
&=\frac{(2i+1-2q+r)^2}{4q(q-r)}-\frac{(2i+1-q)^2}{4r(q-r)}+d(L(r,s), i\mod r)\\
&=-\left(\frac{(2i+1-q+r)^2}{4qr}-\frac{1}{4}-d(L(r,s), i\mod r)\right)
\end{align*}

\end{proof}
Recall $\phi_{a,b}:  \mathbb{Z}/p\mathbb{Z} \to \mathbb{Z}/p\mathbb{Z}$ is defined by $$\phi_{a,b}(i)=ai+b \mod p.$$
\begin{lem}
When $p > 52$,  there is at most one value for $b$ in $\{b_0,b_1\}$.
\label{bvalue}
\end{lem}
\begin{proof}
For a $\mathbf{T}$- or $\mathbf{I}$-type $p$-surgery on a knot, by lemma \ref{Lemma}, $S^3_K(p)\cong \varepsilon \mathbb{T}(p/q)$, where $p=6q+\zeta r$, $r=3\text{ or } 5$.
Here $p$ is odd, $\frac{p}{2}$ is a half integer. By lemma \ref{b}, if $q$ is odd, $b=\frac{q-1}{2}$; if $q$ is even,  $b=\frac{p+q-1}{2}$. We may write $b=\frac{\bar{\theta}(q)p+q-1}{2}$.

For an $\mathbf{O}$-type $p$-surgery on a knot, by lemma \ref{Lemma}, $S^3_K(p)\cong \varepsilon \mathbb{T}(p/q)$, where $p=6q+\zeta r$, $r=4$. Note here $p$ is even, $(p,q)=(p,a)=1$, $q, a$ are odd, so $q=4l+s$, where $s=1,3$. By lemma \ref{b}, $b_j=\frac{jp+q-1}{2}$, $j = 0,1$, and both of them are integers. Denote $\phi_{a,j}(i)=ai+b_j$.

More specifically, $p=6q+\zeta r$, $r=4$, $S^3_K(p)\cong \varepsilon \mathbb{T}(p/q)$, $q=4l+s$,  $\zeta, \varepsilon \in \{1,-1\}$, $s \in \{1,3\}$. For $\phi_{a,j}$, $\phi_{a,j}(0) = \frac{jp+q-1}{2}, \phi_{a,j}(\frac{p}{2}) =\frac{(1-j)p+q-1}{2}$.

Using ($\ref{recursion}$) and (\ref{main_case}), we get 
\begin{align*}
&d(L(p,1),0)-d(L(p,1),\frac{p}{2})=\frac{p}{4}\\
&d(L(p,q),\frac{q-1}{2})-d(L(p,q),\frac{p+q-1}{2})\\
&=\frac{p^2}{4pq}-d(L(q,\frac{1-\zeta}{2}q+\zeta r), \frac{q-1}{2})+d(L(q,\frac{1-\zeta}{2}q+\zeta r), \frac{q+\zeta r-1}{2})\\
&=\frac{p^2}{4pq}-\zeta [\frac{r^2}{4qr}-d(L(4,s),2l+\frac{s-1}{2} \mod 4)+d(L(4,s),2l+2+\frac{s-1}{2} \mod 4)] \\
&=\frac{3}{2} + \zeta (-1)^l
\end{align*}

Here we require $q-r>r$ and $\frac{q+r-1}{2}<q$, it suffices to take $p> 52 = 6*2r+r$.

Using Propsition \ref{prop}, (\ref{T32complement}) and ($\ref{Delta}$), we get
\begin{align*}
\Delta^\varepsilon_{a, b_j}(0) - \Delta^\varepsilon_{a, b_j}(\frac{p}{2})
=\frac{p}{4}\mp\varepsilon (\frac{3}{2} +\zeta (-1)^l-2)=6l+\zeta \mp\varepsilon \zeta(-1)^l+\frac{3}{2}(s\mp \varepsilon)\pm2\varepsilon
\end{align*}

The parity of $\Delta^\varepsilon_{a, b_j}(0) - \Delta^\varepsilon_{a, b_j}(\frac{p}{2})$ depends only on the parity of $\frac{3}{2}(s\mp \varepsilon)$, and by ($\ref{Delta}$), it should be even, so we get 
\begin{align*}
b=\left\{ 
  \begin{array}{l l}
   \frac{q-1}{2} \quad \text{if  } s=1, \varepsilon=1 \quad \text{or} \quad s=3,  \varepsilon=-1\\
   \frac{p+q-1}{2} \quad \text{if  } s=1, \varepsilon=-1 \quad \text{or} \quad s=3,  \varepsilon=1
  \end{array} \right.
\end{align*}
We can write $b=\frac{\bar \theta(\frac{s+\varepsilon}{2})p+q-1}{2}$ for $p>52$.
\end{proof}

Because of Lemma \ref{bvalue}, we can treat $\mathbf{T}$-,$\mathbf{O}$- and  $\mathbf{I}$-type manifolds uniformly. Let 
$$\theta = \theta (q, \varepsilon)=\left\{ 
 \begin{array}{l l}
\bar{\theta} (q) \quad \quad \text{if  } r =3,5\\
\bar \theta(\frac{s+\varepsilon}{2}) \quad \text{if  } r=4, q=4l+s
 \end{array} \right.,$$
 then $b=\frac{\theta p+q-1}{2}$,
we may denote $\phi_{a,b}$ by $\phi_{a,\theta}$.

\begin{lem}
Assume that $S^3_K(p) \cong \varepsilon \mathbb{T}(p/q)$. Let $m \in \{ 0,1,2,3\}$ satisfy that
$$
0\leq a-mq+\frac{\theta \zeta r+q-1}{2}<q,
$$
then
$$
|a-\frac{mp}{6}|<\sqrt{\frac{11rp}{6}}.
$$
\label{lemma4.4}
\end{lem}

\begin{proof}
 By (\ref{ts_nonincreasing}), $\Delta^\varepsilon_{a, \theta}(0) - \Delta^\varepsilon_{a, \theta}(1) = 0 \text{ or } 2$. 
Let $h=\left\{ 
  \begin{array}{l l}
   0 \quad \text{if  } 0\leq \frac{\theta p + q-1}{2}+a < p\\
   1 \quad \text{if  } \frac{\theta p + q-1}{2}+a\geq p
  \end{array} \right.$.

\begin{align}
\label {eq12}
&\Delta^\varepsilon_{a, \theta}(0) - \Delta^\varepsilon_{a, \theta}(1) \\ 
   =& d(L(p,1),0)-\varepsilon [d(L(p,q), \frac{\theta p + q-1}{2})-2\chi_{[0,q)}( \frac{\theta p + q-1}{2})] \notag \\
&-d(L(p,1),1)+\varepsilon[d(L(p,q),  \frac{\theta p + q-1}{2}+a)-2\chi_{[0,q)}( \frac{\theta p + q-1}{2}+a-hp)]\notag\\
=& 2\varepsilon\{\chi_{[0,q)}( \frac{\theta p + q-1}{2})-\chi_{[0,q)}( \frac{\theta p + q-1}{2}+a-hp)\}+\frac{p^2}{4p}-\frac{(p-2)^2}{4p}\notag\\
&-\varepsilon\{ \frac{[(\theta-1) p]^2-pq}{4pq}-d(L(q, \frac{1-\zeta}{2}q+\zeta r),\frac{\theta\zeta r+q-1}{2})\notag\\
&-  \frac{[2a+(\theta-1) p]^2-pq}{4pq}+d(L(q, \frac{1-\zeta}{2}q+\zeta r),\frac{\theta\zeta r+q-1}{2}+a-mq)\}\notag
\end{align}

Let $i = \frac{\theta\zeta r+q-1}{2} \mod r, j =  \frac{\theta\zeta r+q-1}{2}+a-mq \mod r$.
Since $0\leq \frac{\theta \zeta r+q-1}{2}+a-mq<q$, we use (\ref{main_case}), the right hand side of (\ref{eq12}) becomes

\begin{align*}
& 2\varepsilon\left(\chi_{[0,q)}( \frac{\theta p + q-1}{2})-\chi_{[0,q)}( \frac{\theta p + q-1}{2}+a-hp)\right)+\frac{p-1}{p}+\varepsilon\frac{a[a+(\theta-1)p]}{pq}\\
&+\varepsilon\zeta \left( \frac{[(\theta-1) \zeta r]^2-qr}{4qr}-d(L(r,s),i)- \frac{[2a-2mq+(\theta-1)\zeta r]^2-qr}{4qr}+d(L(r,s),j)\right)\\
&= C+\varepsilon\frac{a[a+(\theta-1)p]}{pq} -\varepsilon\zeta \frac{[a-mq+(\theta-1) \zeta r](a-mq)}{qr}\\
&=-\frac{6\varepsilon\zeta}{pr}(a - \frac{mp}{6})^2  -\varepsilon m (1 -\theta) +\frac{ \varepsilon m^2}{6} +C,
\end{align*}
where
$$C=  2\varepsilon\{\chi_{[0,q)}( \frac{\theta p + q-1}{2})-\chi_{[0,q)}( \frac{\theta p + q-1}{2}+a-hp)\}+\varepsilon \zeta [d(L(r,s),j)-d(L(r,s),i)]+\frac{p-1}{p}.$$

Using (2),(3),(4), $|C|\leq \frac{6}{5}+2+1<\frac{9}{2}$. 

Moreover, $|-\varepsilon m (1 -\theta) +\frac{ \varepsilon m^2}{6}|\leq m+\frac{m^2}{6}\leq 3+\frac{3}{2}=\frac{9}{2}$.
So we get 
$$
|\frac{6}{pr}(a - \frac{mp}{6})^2|< 2+\frac{9}{2}+\frac{9}{2}=11,
$$
so our conclusion holds.
\end{proof}

\begin{lem}
Suppose $p>767$.
Let $k$ be an integer satisfying 
\begin{equation}
0 \leq k<\frac{1}{6}\left(\frac{\sqrt 6}{13 \sqrt{11r}}\sqrt p-1\right).
\label{k_condition}
\end{equation}
Let $$i_k=\frac{\theta\zeta r+q-1}{2}+6ka-kmp \mod r, j_k=\frac{\theta\zeta r+q-1}{2}+(6k+1)a-kmp-mq \mod r.$$
Then 
$\Delta^\varepsilon_{a, \theta}(6k) - \Delta^\varepsilon_{a, \theta}(6k+1) = Ak+B+C_k$,
where 
\begin{align*}
A &=-\frac{72\varepsilon\zeta}{pr}(a - \frac{mp}{6})^2  - \frac{12}{p},\\ 
B&=\varepsilon\left(-\frac{6\zeta}{pr}(a - \frac{mp}{6})^2  -m (1 -\theta) +\frac{ m^2}{6}\right)\\
&+ 2\varepsilon\{ 2\varepsilon\{\chi_{[0,q)}( 3\theta q)-\chi_{[0,q)}( (3\theta+m-6h) q)\}+\frac{p-1}{p},\\
C_k&=\varepsilon \zeta [d(L(r,s),j_k)-d(L(r,s),i_k)].
\end{align*}
and 
\begin{align*}
h=\left\{ 
  \begin{array}{l l}
   0 \quad \text{if  } 0\leq 3\theta + m < 6\\
   1 \quad \text{if  } 3\theta + m = 6
  \end{array} \right..
\end{align*}
\end{lem}

\begin{proof}
 
Using(\ref{delta_plus}), we get
\begin{align}
\label{eq14}
&\Delta^\varepsilon_{a, \theta}(6k) - \Delta^\varepsilon_{a, \theta}(6k+1) \\ 
 =& d(L(p,1),6k) -d(L(p,1),6k+1)-\varepsilon [d(L(p,q), \frac{\theta p + q-1}{2}+6ka-kmp) \notag \\
&-2\chi_{[0,q)}( \frac{\theta p + q-1}{2}+6ka-kmp)] +\varepsilon[d(L(p,q),  \frac{\theta p + q-1}{2}+(6k+1)a-kmp) \notag\\
&-2\chi_{[0,q)}( \frac{\theta p + q-1}{2}+(6k+1)a-(km+h)p)]\notag\\
=& 2\varepsilon\{\chi_{[0,q)}( \frac{\theta p + q-1}{2}+6ka-kmp)-\chi_{[0,q)}(  \frac{\theta p + q-1}{2}+(6k+1)a-(km+h)p)\} \notag \\
&+\frac{(p-12k)^2}{4p}-\frac{[p-2(6k+1)]^2}{4p}-\varepsilon\{ \frac{[12ka-(2km+1-\theta) p]^2}{4pq} \notag \\
&- \frac{[(12k+2)a-(2km+1-\theta) p]^2}{4pq}-d(L(q, \frac{1-\zeta}{2}q+\zeta r),\frac{\theta\zeta r+q-1}{2}+6ka-kmp) \notag \\
&+d(L(q, \frac{1-\zeta}{2}q+\zeta r),\frac{\theta\zeta r+q-1}{2}+(6k+1)a-kmp-mq)\} \notag
\end{align}

We require 
\begin{align*}
0 \leq&\frac{\theta \zeta r+ q-1}{2}+6ka-kmp<q,\\
0\leq &\frac{\theta\zeta r+q-1}{2}+(6k+1)a-kmp-mq< q.
\end{align*} It suffices that 
$$k<\frac{1}{6}\left(\frac{q-9}{2}\sqrt{\frac{6}{11rp}}-1\right).$$
This implies 
\begin{align}
3\theta q &\leq\frac{\theta p+ q-1}{2}+6ka-kmp<(3\theta +1)q, \label{14}\\
(3\theta+m) q&\leq \frac{\theta p +q-1}{2}+(6k+1)a-kmp<(3\theta+m+1) q. \label{15}
\end{align}
When $m=3$, $\theta = 1$, (\ref{15}) becomes 
$$6 q\leq \frac{\theta p +q-1}{2}+(6k+1)a-kmp<7 q.$$
Here we require $$p\leq \frac{\theta p+q-1}{2}+(6k+1)a-kmp< p+q.$$ 
We know $a<\frac{p}{2}$, so 
$$\frac{\theta p+q-1}{2}+(6k+1)a-kmp< \frac{\theta p+q-1}{2}+\frac{p}{2} < p+q.$$ 
Moreover, we know when $m = 3$, by Lemma \ref{lemma4.4}, $a > \frac{p}{2}-\sqrt{\frac{11rp}{6}}$.
 If $$k\leq \frac{1}{6}\left(\frac{q-1}{2}\sqrt{\frac{6}{11rp}}-1\right),$$ then
$$\frac{\theta p+q-1}{2}+(6k+1)a-kmp \geq p.$$
When $p>767$, $$\frac{1}{6}\left(\frac{p}{13}\sqrt{\frac{6}{11rp}}-1\right)<\frac{1}{6}\left(\frac{q-9}{2}\sqrt{\frac{6}{11rp}}-1\right).$$

Using (\ref{main_case}), (\ref{14}) and (\ref{15}), the right hand side of (\ref{eq14}) becomes
\begin{align*}
& 2\varepsilon\{\chi_{[0,q)}( 3\theta q)-\chi_{[0,q)}( (3\theta+m-6h) q)\}+\frac{p-(12k+1)}{p}\\
&+\varepsilon\frac{a((12k+1)a-2kmp+(\theta-1)p)}{pq}\\
&+\varepsilon\zeta \{ \frac{[12ka-2kmp+(\theta-1)\zeta r]^2-qr}{4qr}-d(L(r,s),i_k)\\
&- \frac{[2(6k+1)a-2kmp-2mq+(\theta-1)\zeta r]^2-qr}{4qr}+d(L(r,s),j_k)\}.
\end{align*}
This simplifies to be
\begin{align*}
& 2\varepsilon\{\chi_{[0,q)}( 3\theta q)-\chi_{[0,q)}( (3\theta+m-6h) q)\}+\varepsilon \zeta [d(L(r,s),j_k)-d(L(r,s),i_k)]+\frac{p-(12k+1)}{p}\\
&+\varepsilon\frac{a((12k+1)a-2kmp+(\theta-1)p)}{pq} -\varepsilon\zeta \frac{((12k+1)a-2kmp-mq+(\theta-1) \zeta r)(a-mq)}{qr}\\
=&-\frac{6(12k+1)\varepsilon\zeta}{pr}(a - \frac{mp}{6})^2  -\varepsilon m (1 -\theta) +\frac{ \varepsilon m^2}{6} + 2\varepsilon\{\chi_{[0,q)}( 3\theta q)-\chi_{[0,q)}( (3\theta+m-6h) q)\}\\
&+\varepsilon \zeta [d(L(r,s),j_k)-d(L(r,s),i_k)]+\frac{p-(12k+1)}{p}\\
=& Ak+B+C_k.
\end{align*}

\end{proof}

\begin{proof} [Proof of Proposition 4.1]
If $S^3_K(p)\cong \varepsilon \mathbb{T}(p/q)$, then (\ref{Delta}) holds, so 
\begin{equation}
\Delta^\varepsilon_{a, \theta}(6k) - \Delta^\varepsilon_{a, \theta}(6k+1)  = 0 \text{ or } 2
\label{0 or 2}
\end{equation}
for all $k$ satisfying (\ref{k_condition}). If $p>310r(36r+1)^2$, then
$$
6\cdot6r +1<\frac{\sqrt 6}{13 \sqrt{11r}}\sqrt p
$$
hence $k = 6r$ satiffies (\ref{k_condition}).

Let $A, B, C_k$ be as in Lemma 4.4. If $A \neq 0$, then $Ak+B+C$ is equal to 0 or 2 for at most two values of $k$ for any given $C$. Given $p,q,a,\varepsilon, \zeta$, as $k$ varies, $C_k$ can take at most $3r$ values. It follows that $Ak+B+C_k$ cannot be 0 or 2 for $k =0, 1, ..., 6r$. As a consequence, if $p>310r(36r+1)^2$, then (\ref{0 or 2}) does not hold.

The only case we need to consider is that $A=0$. In this case $\varepsilon\zeta = -1$. $$A= \frac{12}{p}\left(\frac{6}{r}(a - \frac{mp}{6})^2 -1\right)=0.$$
We get $|a-\frac{mp}{6}| = \sqrt{\frac{r}{6}}$, which is an irrational number. This contradicts that $a$ is an integer and $\frac{mp}{6}$ is a rational number.
\end{proof}

Since we get an upper bound for $p$, an easy computer search will yield all possible $p/q$'s. They are 1/1, 2/1, 3/1, 7/2, 9/1, 9/2, 10/1, 10/1, 11/1, 13/3, 13/3, 14/3, 17/2, 17/2, 19/4, 21/4, 22/3, 23/3, 27/4, 27/5, 29/4, 29/4, 37/7, 37/7, 38/7, 43/8, 46/7, 47/7, 49/9, 50/9, 51/8, 58/9, 59/9, 62/11, 69/11, 70/11, 81/13, 81/14, 83/13, 86/15, 91/16, 93/16, 94/15, 99/16, 99/17, 101/16, 106/17, 106/17, 110/19, 110/19, 113/18, 113/18, 119/19, 131/21, 133/23, 137/22, 137/22, 143/23, 146/25, 154/25, 157/27, 157/27, 163/28, 211/36, 221/36. Here if $p_i/q_i$ appears twice, this means they correspond to candidate knots with different Heegaard Floer Homologies.

\section{Berge knots with $\mathbf{T}$-,$\mathbf{O}$- and $\mathbf{I}$-type surgeries}
\begin{prop}
Only 11 hyperbolic berge knots have $\mathbf{T}$-,$\mathbf{O}$- and $\mathbf{I}$-type surgeries. More precisely, let $K(p,q; \lambda)$ be the berge knot corresponding to homology class $\lambda$ in $L(p,q)$. They are $K(18,5;5)$, $K(39,16;16)$, $K(45,19;8)$, $K(46,19;11)$, $K(68,19;5)$, $K(71,27;11)$, $K(82,23;5)$, $K(93,26;5)$, $K(107,30;5)$, $K(118,33;5)$, $K(132,37;5)$.
\end{prop}

\begin{proof}
Berge knots have lens space surgeries, we can compute their Heegaard Floer Homology using Proposition \ref{prop} and compare them with the list from finite surgeries, we get 11 candidates. We draw the link diagram in SnapPy and compute the fundamental group of the Dehn filling with the finite surgery coefficient. If what we get is a $(2,3,n)$-type group, then we have verified that the knot has indeed a required finite surgery. Below is a table of the candidates.

\begin{table}[ht]
\caption{Canditaes of Berge knots with $\mathbf{T}$-,$\mathbf{O}$- and $\mathbf{I}$-type surgeries}  
\centering    

\begin{tabular}{|c| c| c| c|}  
\hline  
 p  & q  & $\lambda$ & finite surgery coefficient p'\\ [0.5ex] 
\hline 
18 & 5 & 5 & 17 \\\hline
39 & 16 & 16 & 38 \\\hline 
45 & 19 & 8 & 46 \\\hline
46 & 17 & 11 & 47 \\\hline
68 & 19 & 5 & 69 \\\hline
71 & 21 & 11 & 70 \\\hline
82 & 23 & 5 & 81 \\\hline
93 & 25 & 5 & 94 \\\hline
107 & 25 & 5 & 106 \\\hline
118 & 25 & 5 & 119 \\\hline
132 & 25 & 5 & 131 \\ 
\hline   
\end{tabular}
\label{table:candidates} 
\end{table}

Here we use a point of view of dual berge knots in the corresponding lens space, as they have the same knot complement as berge knots in $S^3$, the only thing is to figure out the corresponding coefficients. The computation is as follows. 

\begin{figure}[h]
        \centering
\includegraphics[width=0.8\textwidth]{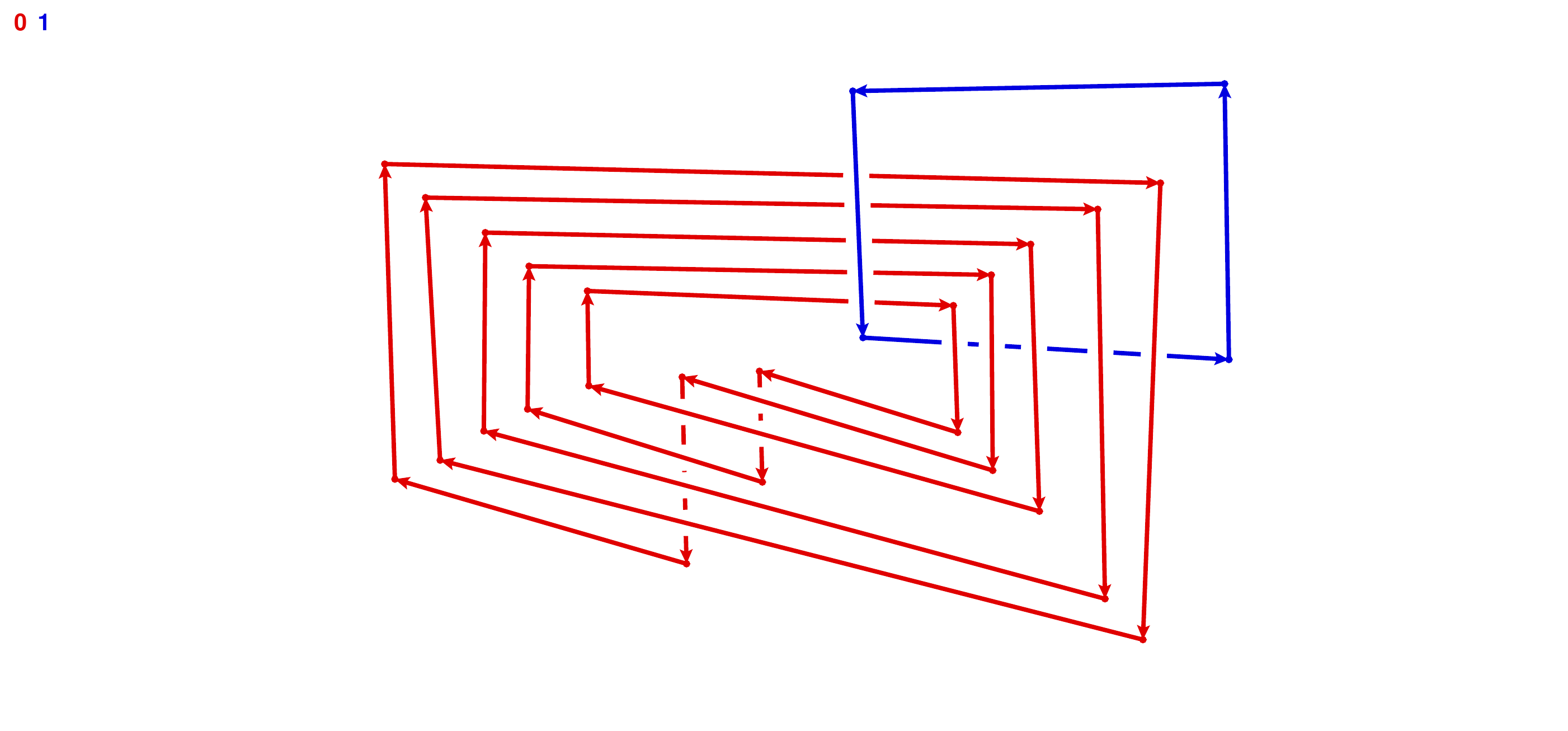}
\caption{Link for $K(107,30;5)$}
\end{figure}

We would like to draw a link $L$ with two components as above, and denote the link complement by $M$.  When we perform on the trivial component ($1$-component) $p/q$-surgery, we get $L(p,q)$. The other component then becomes the dual berge knot in $L(p,q)$. We would like to choose the orientations of two components consisistently, one choice is shown above, and the other choice is to reverse the orientations of both two components. Denote the longitude and meridian of $0$-component by $\lambda$ and $\mu$, and those of $1$-component by $l$ and $m$. In homology, we have $w\mu=l$ and $\lambda=wm$. After performing $p/q$-surgery on $1$-component, we have $pm+ql=0$ in homology and we denote the longitude and meridian of resultant $0$-component $K'$ by $\lambda'$ and $\mu'$. We have $qw^2\mu+p\lambda=qwl+pwm=0$, this means that $\lambda'=(qw^2, p)$. By performing $p$-surgery on $K'$ we get $L(p,q)$ (this corresponds to performing $\infty$-surgery on $0$-component of $L$), and by performing $\infty$-surgery on $K'$ we get $S^3$, and we hope to get a spherical space form by performing $p'$-surgery on $K'$. We have $\lambda'=qw^2\mu+p\lambda=(qw^2,p)$, and $p\mu'+\lambda'=(\pm 1, 0)$.  Note that $L(p,q_1) \cong L(p, q_2)$ iff $q_1q_2\equiv \pm 1 \, \text{mod} \, p$, so there is an indetermincy in $q$. Take $K(107,25;5)$ for example. We have  $\lambda'=(25*5^2,107)$, and $107\mu'+\lambda'=(\pm 1, 0)$. There is no integer solutions. Instead we take $q=30$, as $\text{mod}(30*25,107)=1$, denote the knot by $K(107,30;5)$ from now on. We thus have $\lambda'=(30*5^2,107)$, $\mu'=(-7,-1)$, and $107\mu'+\lambda'=(1, 0)$. $ M(1/0,107/30)$ is $L(107,30)$, and $ M(7/1,107/30)$ is $S^3$. And $M(8/1,107/30)$ has fundamental group of the form  $<a,b|abab^{-2},ab(b^3a^2)^{77}b^4ab((a^{-2}b^{-3})^3a^{-1})^{21}a^{-1}b>$. The first relator is $(ab)^2=b^3$, which is in the center. Mod it out, the other relator becomes $a^4=1$. By upper central series theory, this group is finite and is of type $(2,3,4)$, which means this dehn filling is indeed the required finite surgery.

\begin{table}[ht]
\caption{Berge knots with $\mathbf{T}$-,$\mathbf{O}$- and $\mathbf{I}$-type surgeries}  
\centering    

\begin{tabular}{|c| c| c| c| c| c|}  
\hline 
p  &q  &$\lambda$ & finite surgery coefficient $p'$ &$Z(\pi_1(S^3_K(p')))$ & $\pi_1(S^3_K(p'))/Z(\pi_1(S^3_K(p')))$ \\ [0.5ex] 
\hline  
18 & 5 & 5 & 17 &  $\langle (ab)^2\rangle$ & $\langle a,b|(ab)^2=b^3=a^5=1\rangle$\\  \hline
39 & 16 & 16 & 38 & $\langle a^4\rangle$ & $\langle a,b|b^2=(a^{-1}b)^3=a^4=1\rangle$\\  \hline
45 & 19 & 8 & 46 & $\langle (ab)^2\rangle$ & $\langle a,b|(ab)^2=a^3=b^4=1\rangle$\\  \hline
46 & 17 & 11 & 47 & $\langle b^3\rangle$ & $\langle (b^2a^{-2})^2 = b^3 =a^5 =1 \rangle$\\  \hline
68 & 19 & 5 & 69  & $\langle (ab)^2\rangle$ & $\langle a,b|(ab)^2=b^3=a^3=1\rangle$ \\  \hline
71 & 27 & 11 & 70  & $\langle a^3\rangle$ & $\langle a,b|b^2=a^3=(ba^2)^4=1\rangle$\\  \hline
82 & 23 & 5 & 81  & $\langle (ab)^2\rangle$ & $\langle a,b|(ab)^2=b^3=a^3=1\rangle$\\  \hline
93 & 26 & 5 & 94  & $\langle (ab)^2\rangle$ & $\langle a,b|(ab)^2=b^3=a^4=1\rangle$\\  \hline
107 & 30 & 5 & 106 & $\langle (ab)^2\rangle$ & $\langle a,b|(ab)^2=b^3=a^4=1\rangle$ \\  \hline
118 & 33 & 5 & 119  & $\langle (ab)^2\rangle$ & $\langle a,b|(ab)^2=b^3=a^5=1\rangle$\\  \hline
132 & 37 & 5 & 131  & $\langle (ab)^2\rangle$ & $\langle a,b|(ab)^2=b^3=a^5=1\rangle$\\ 
\hline   
\end{tabular}
\label{table:berge knots} 
\end{table}
\end{proof}

\section{Summary}
I summarize all results below with a table of torus knots and satellite knots and a table of hyperbolic knots. Let $T(p,q)$ be the $(p,q)$ torus knot, and let $[p_1,q_1; p_2, q_2]$ denote the $(p_1, q_1)$ cable of $T(p_2, q_2)$. I have also drawed all hyperbolic knots below using Mathematica package KnotTheory\hspace{0.5 mm}$\grave{}$.

\begin{table}[ht]
\centering

\captionof{table}{torus knots and satellite knots with $\mathbf{T}$-,$\mathbf{O}$- and $\mathbf{I}$-type surgeries}  

\begin{tabular}{|c |c |c |c |c |c|c|c|c|c|}  
\hline  
p & knot & p & knot & p & knot &p & knot & p & knot\\ [0.5ex] 
\hline  
1 & $T(3,2)$ & 2 & $T(3,2)$ &3 & $T(3,2)$ & 7 & $T(5,2)$ & 9 & $T(3,2)$\\ \hline
10 & $T(3,2)$ & 10 & $T(4,3)$ & 11 & $T(3,2)$ & 13 & $T(5,3)$ & 13 & $T(5,2)$\\ \hline
14 & $T(4,3)$ & 17 & $T(5,3)$ & 19 & $[9,2;3,2]$ & 21 & $[11,2;3,2]$& 27 & $[13,2;3,2]$\\ \hline
29 & $[15,2;3,2]$ & 37 & $[19,2;5,2]$ & 43 & $[21,2;5,2]$ & 49 & $[16,3;3,2]$ & 50 & $[17,3;3,2]$\\ \hline
59 & $[20,3;3,2]$ & 91 & $[23,4;3,2]$ & 93 & $[23,4;3,2]$ & 99 & $[25,4;3,2]$ & 101 & $[25,4;3,2]$\\ \hline
106 & $[35,3;4,3]$ & 110 & $[37,3;4,3]$ & 133 & $[44,3;5,3]$ & 137 & $[46,3;5,3]$ & 146 & $[29,5;3,2]$\\ \hline
154 & $[31,5;3,2]$& 157 & $[39,4;5,2]$ & 163 & $[41,4;5,2]$ & 211 & $[35,6;3,2]$ & 221 & $[37,6;3,2]$\\ 
\hline   
\end{tabular}
\label{table: torus knots and satellite knots} 

\captionof{table}{hyperbolic knots with $\mathbf{T}$-,$\mathbf{O}$- and $\mathbf{I}$-type surgeries}  

\begin{tabular}{|c |c| c| c| c| c|}  
\hline 
p & knot \\ [0.5ex] 
\hline  
17 & Preztel knot $P(-2,3,7)$\\  \hline 
22 & Preztel knot $P(-2,3,9)$\\  \hline
23 & Preztel knot $P(-2,3,9)$\\  \hline
29 & mirror image of $K(1,1,0)$ from section 4  of \cite{Eudave_Munoz}\\  \hline
37 & $K(11,3,2,1,1)$ from \cite{Dean}\\  \hline
38 & Berge knot $K(39,16; 16)$\\  \hline
46 & Berge knot $K(45,19; 8)$\\  \hline
47 & $K_3^*$ from \cite{Boyer_Zhang}\\  \hline
51 & $K_2^\sharp$ from \cite{Boyer_Zhang}\\  \hline
58 &  mirror image of the P/SF$_\text{d}$ KIST IV knot with $(n,p,\epsilon, J_1, J_2)=(-2,1,1,-4,1)$ from \cite{Berge2}\\  \hline
62 & P/SF$_\text{d}$ KIST III knot with $(h,k,h', k', J)=(-5,-3,-2,-1,1)$ from \cite{Berge2}\\  \hline
69 & mirror image of $K(2,3,5,1,-3)$ from \cite{Miyazaki_Motegi}\\  \hline
70 & Berge knot $K(71,27;11)$\\  \hline
81 & $K(2,3,5,1,3)$ from \cite{Miyazaki_Motegi}\\  \hline
83 & P/SF$_\text{d}$ KIST V knot with $(n,p,\epsilon, J_1, J_2)=(1,-2,-1,2,2)$ from \cite{Berge2}\\  \hline
86 & mirror image of $K(3,4,7,1,-2)$ from \cite{Miyazaki_Motegi}\\  \hline
94 & mirror image of $K(2,3,5,1,-4)$ from \cite{Miyazaki_Motegi}\\  \hline
106 & $K(2,3,5,1,4)$ from \cite{Miyazaki_Motegi}\\  \hline
110 & $K(3,4,7,1,2)$ from \cite{Miyazaki_Motegi}\\  \hline
113 & mirror image of $K(3,5,8,1,-2)$ from \cite{Miyazaki_Motegi}\\  \hline
113 & mirror image of the P/SF$_\text{d}$ KIST V knot with $(n,p,\epsilon, J_1, J_2)=(-3,-2,-1,2,2)$ from \cite{Berge2}\\  \hline
119 & mirror image of $K(2,3,5,1,-5)$ from \cite{Miyazaki_Motegi}\\  \hline
131 & $K(2,3,5,1,5)$ from \cite{Miyazaki_Motegi}\\  \hline
137 & mirror image of $K(2,5,7,1,-3)$ from \cite{Miyazaki_Motegi}\\  \hline
143 & $K(3,5,8,1,2)$ from \cite{Miyazaki_Motegi}\\  \hline
157 & $K(2,5,7,1,3)$ from \cite{Miyazaki_Motegi}\\
\hline   
\end{tabular}
\label{table: hyperbolic knots} 
\end{table}
\FloatBarrier

\begin{figure}[h]
        \centering

\begin{subfigure}[b]{0.5\textwidth}
	\centering
           \includegraphics[width=\textwidth]{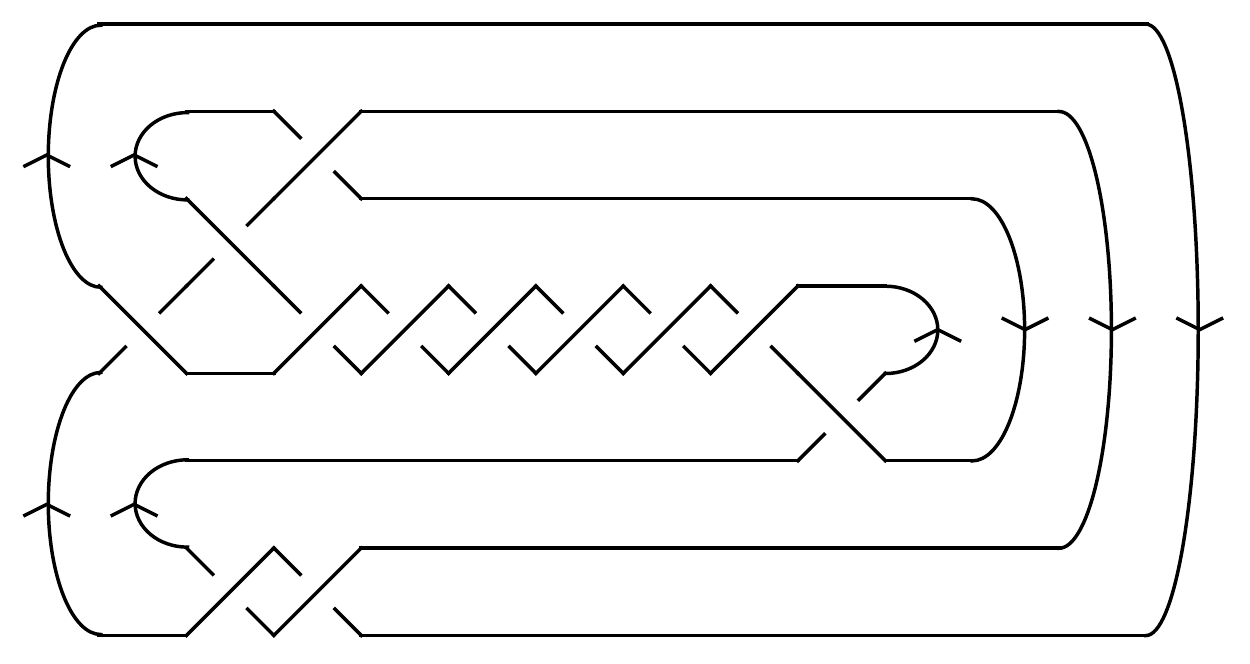}
           \caption*{17}
\end{subfigure}%
      ~
\begin{subfigure}[b]{0.5\textwidth}
	\centering
           \includegraphics[width=\textwidth]{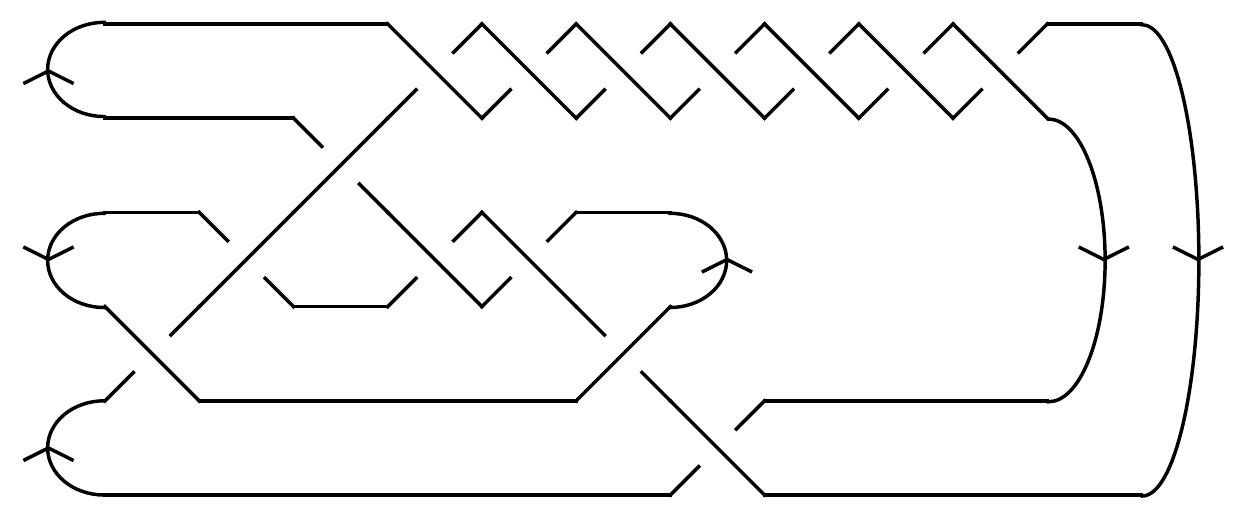}
           \caption*{22,23}
\end{subfigure}

\begin{subfigure}[b]{0.5\textwidth}
	\centering
           \includegraphics[width=\textwidth]{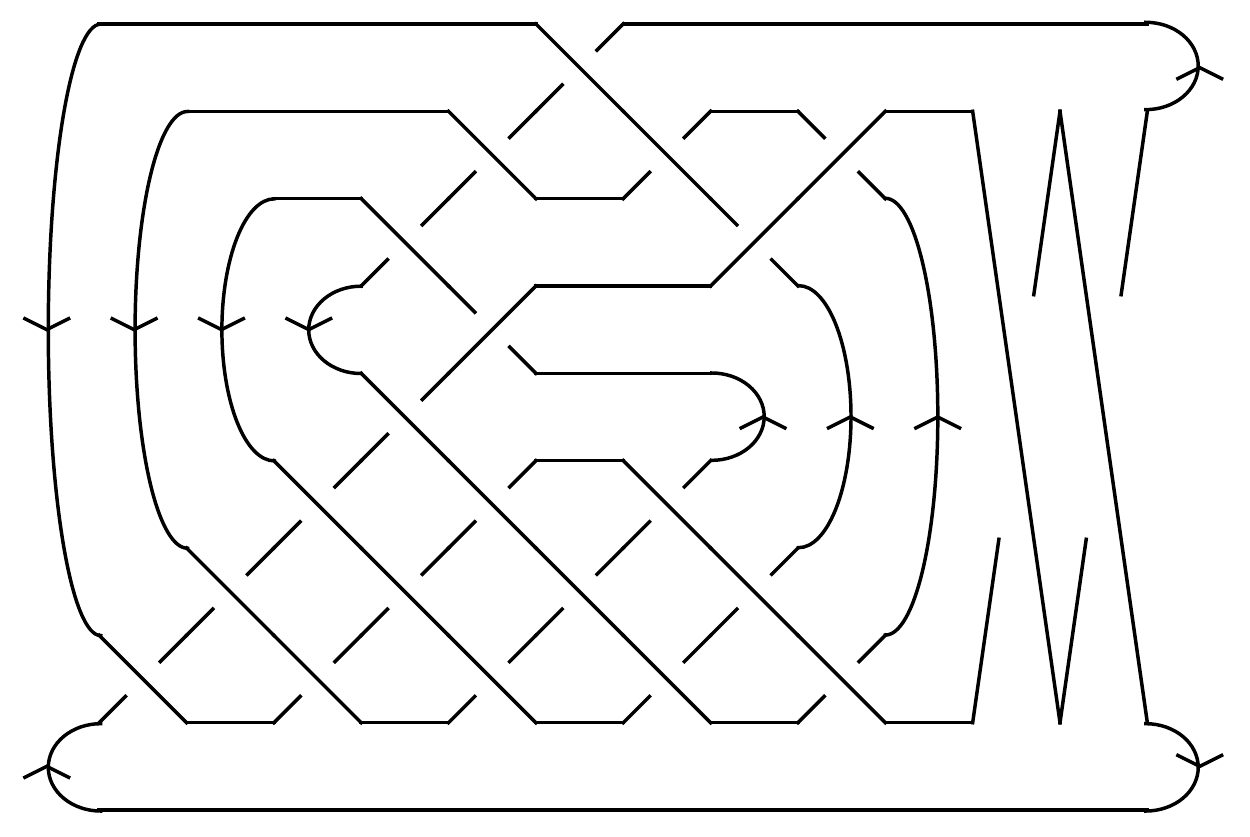}
           \caption*{29}
\end{subfigure}
       ~
\begin{subfigure}[b]{0.5\textwidth}
	\centering
           \includegraphics[width=\textwidth]{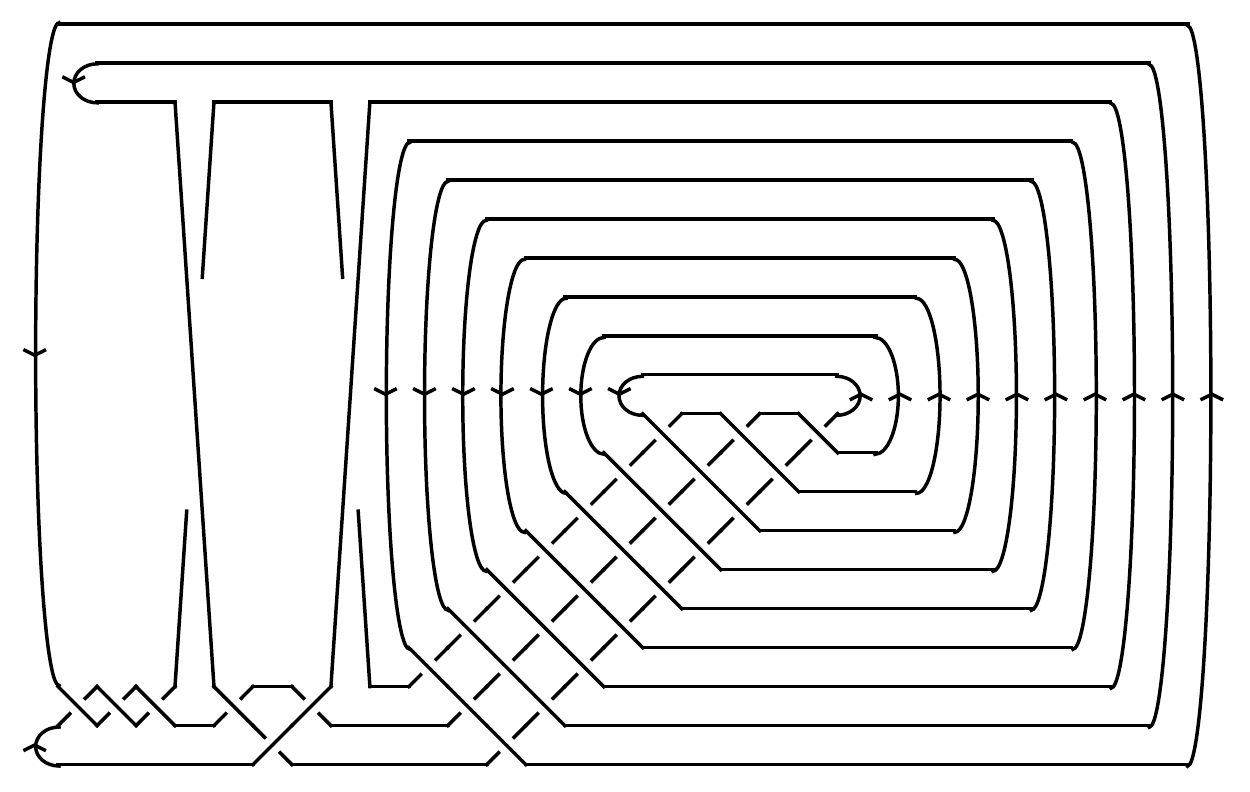}
           \caption*{37}
\end{subfigure}

\begin{subfigure}[b]{0.5\textwidth}
	\centering
           \includegraphics[width=\textwidth]{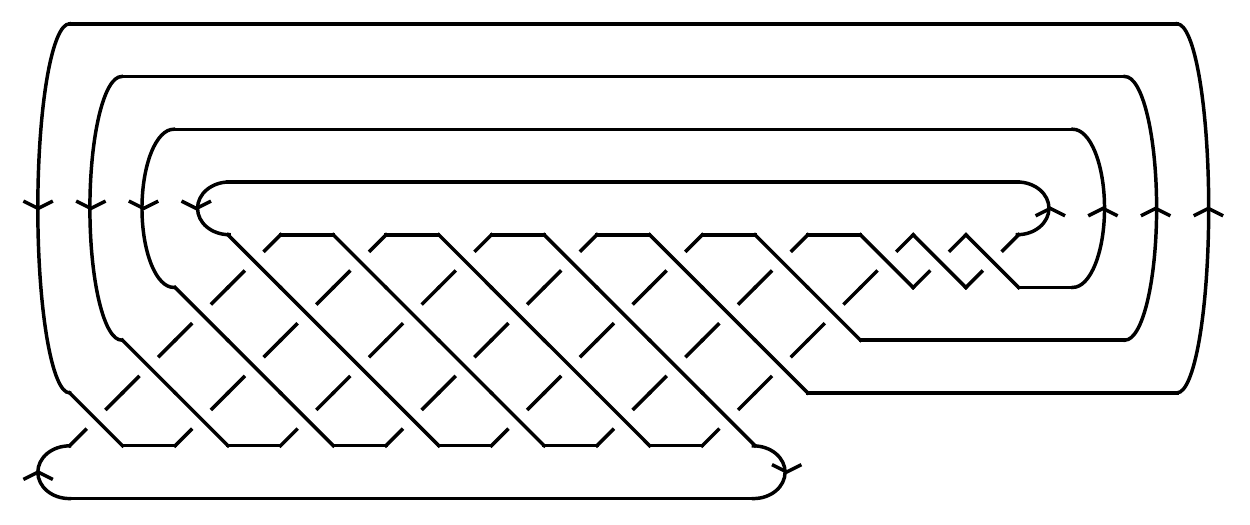}
           \caption*{38}
\end{subfigure}%
      ~
\begin{subfigure}[b]{0.5\textwidth}
	\centering
           \includegraphics[width=\textwidth]{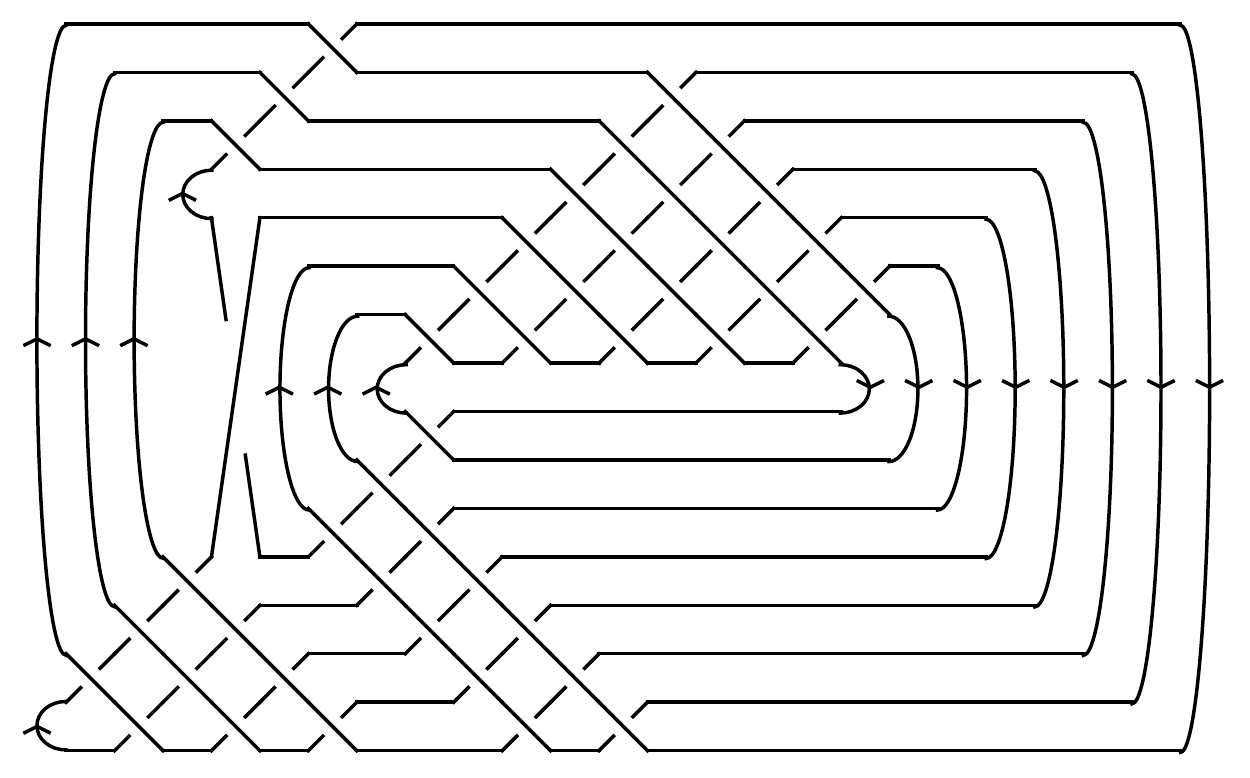}
           \caption*{46}
\end{subfigure}

\begin{subfigure}[b]{0.5\textwidth}
	\centering
           \includegraphics[width=\textwidth]{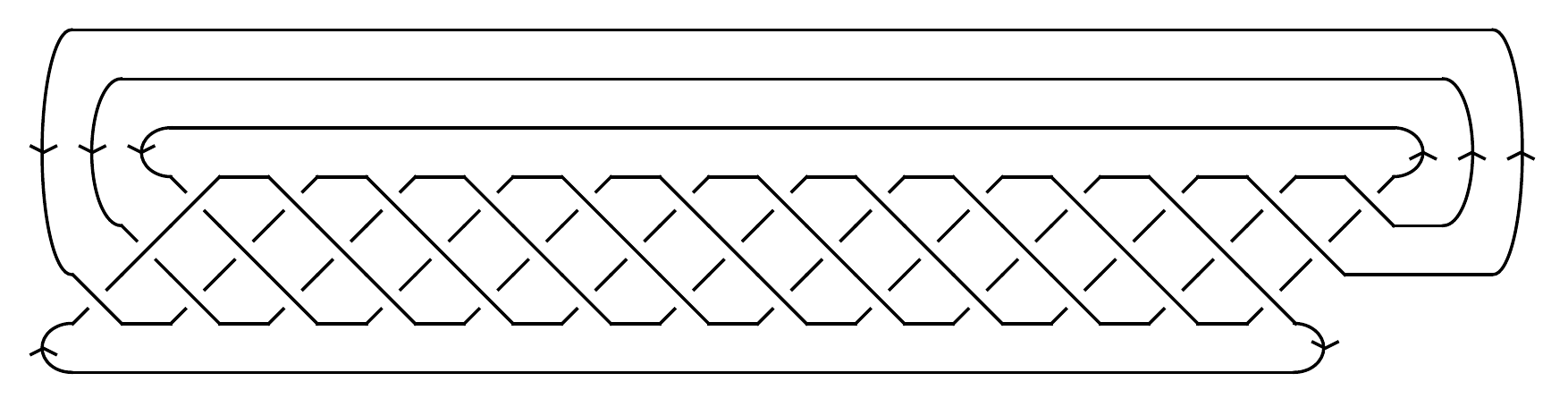}
           \caption*{47}
\end{subfigure}
       ~
\begin{subfigure}[b]{0.5\textwidth}
	\centering
           \includegraphics[width=\textwidth]{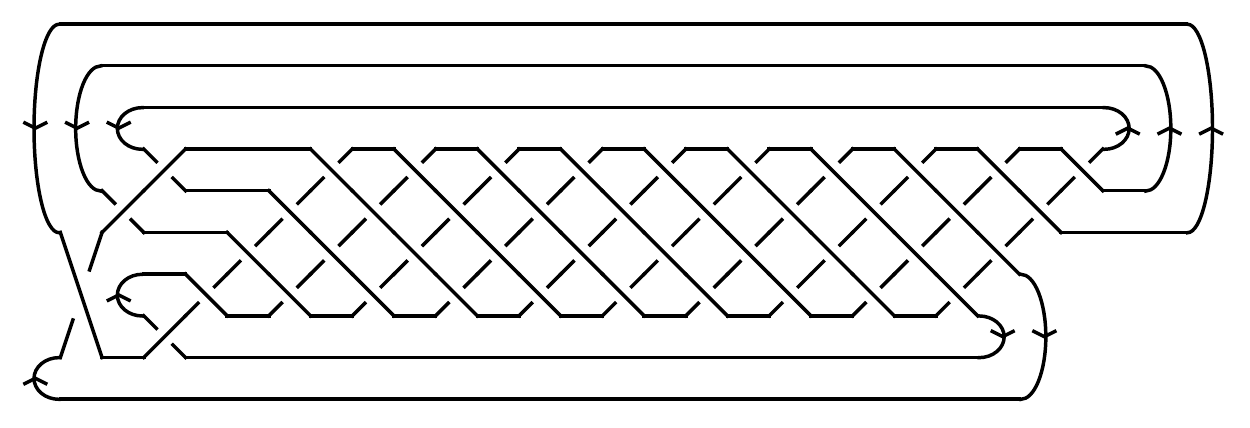}
           \caption*{51}
\end{subfigure}
   
\begin{subfigure}[b]{0.5\textwidth}
	\centering
           \includegraphics[width=\textwidth]{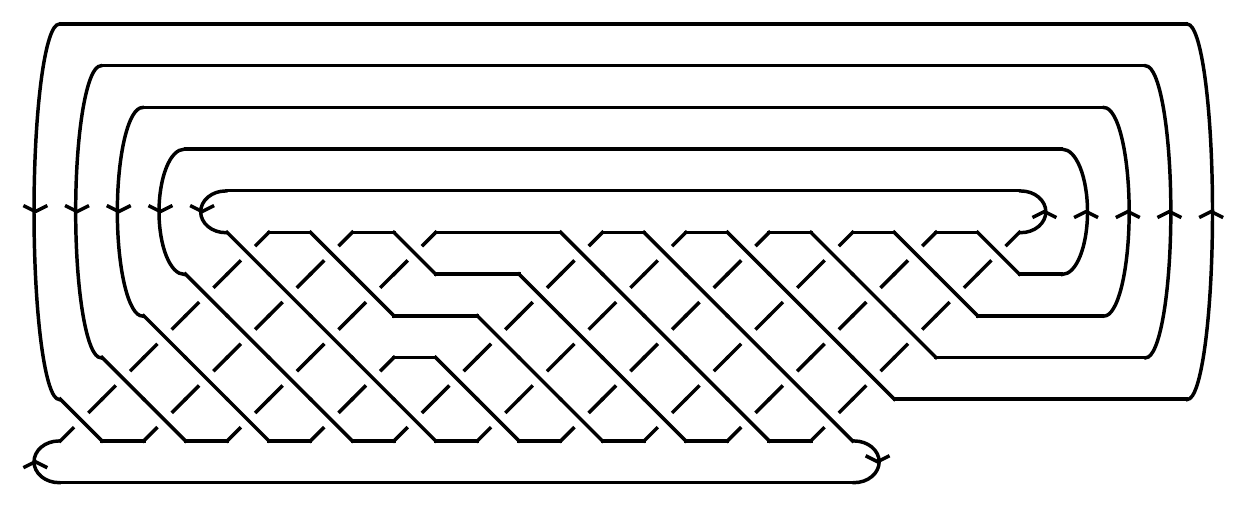}
           \caption*{58}
\end{subfigure}%
      ~
\begin{subfigure}[b]{0.5\textwidth}
	\centering
           \includegraphics[width=\textwidth]{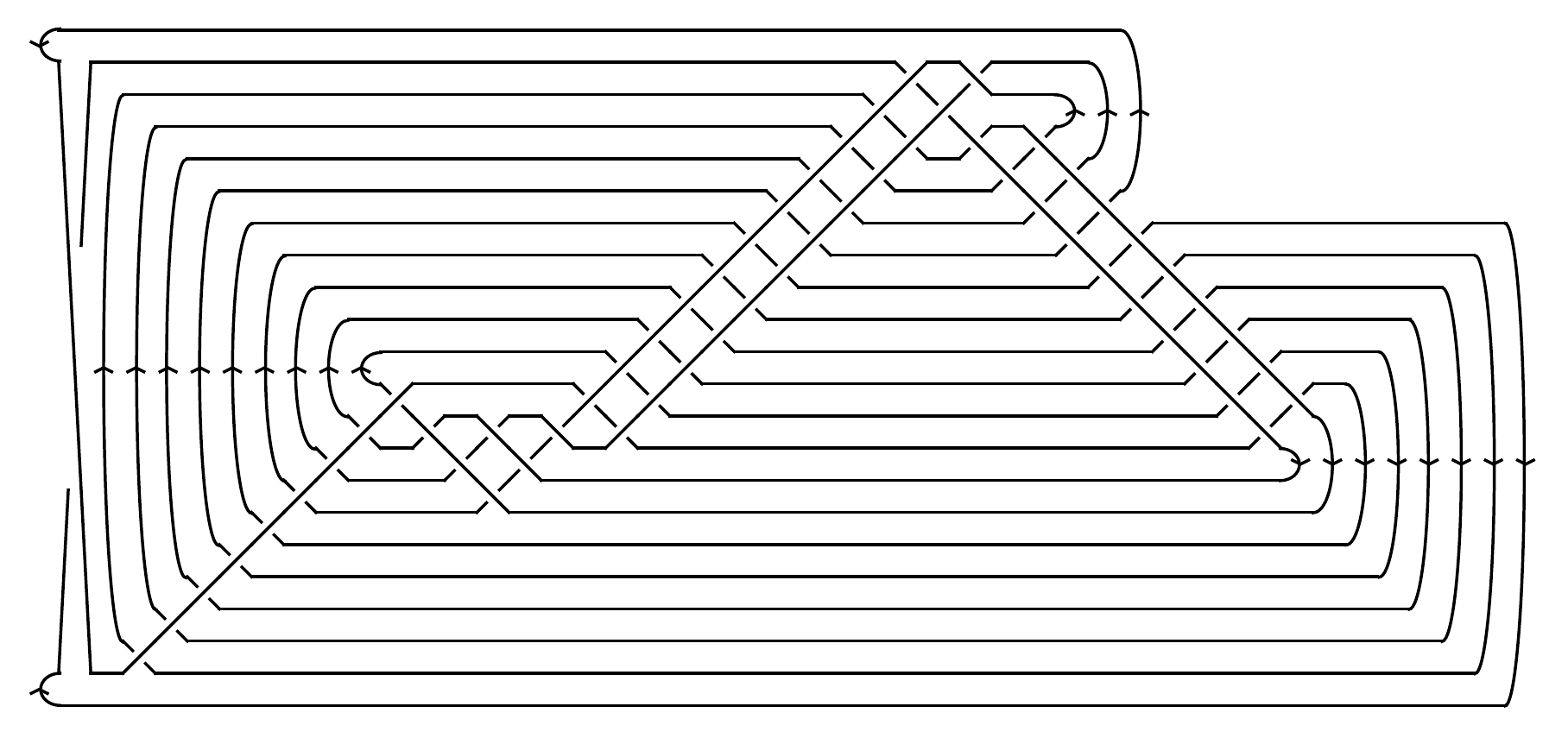}
           \caption*{62}
\end{subfigure}

\begin{subfigure}[b]{0.5\textwidth}
          \centering
          \includegraphics[width=\textwidth]{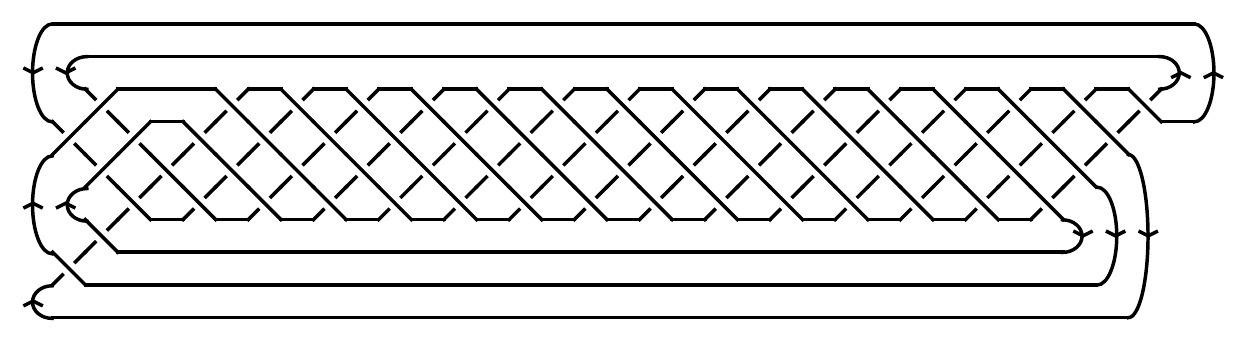}
          \caption*{69}
\end{subfigure}
~
\begin{subfigure}[b]{0.5\textwidth}
         \centering
         \includegraphics[width=\textwidth]{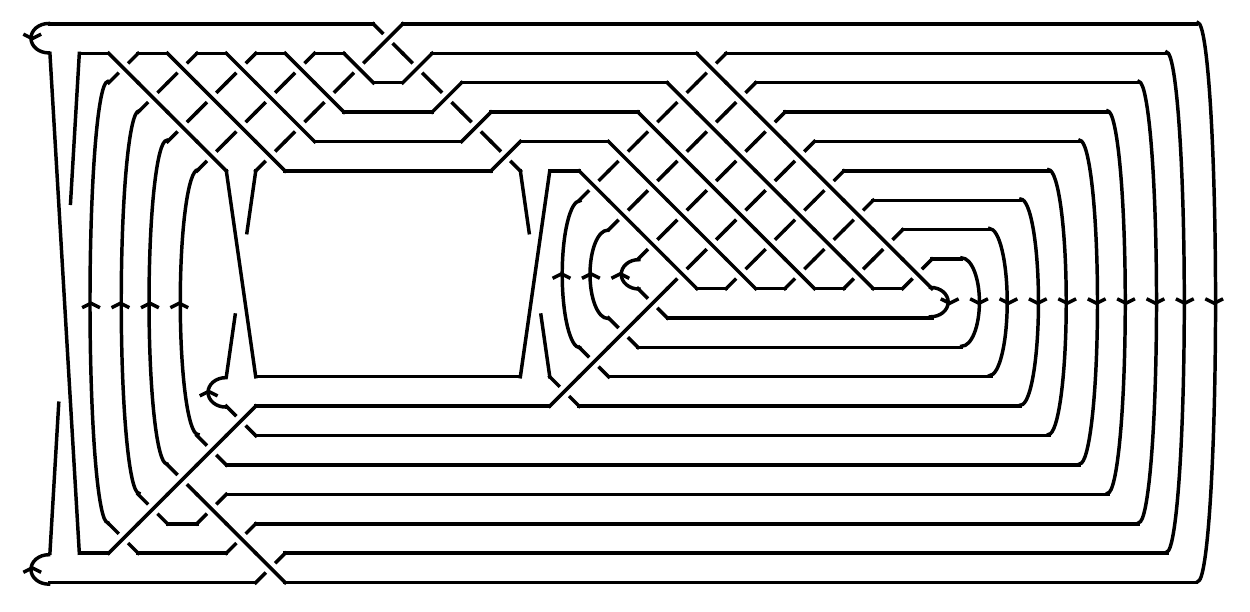}
         \caption*{70}
\end{subfigure}

\end{figure}

\begin{figure}[h]
        \centering

\begin{subfigure}[b]{0.8\textwidth}
          \centering
          \includegraphics[width=\textwidth]{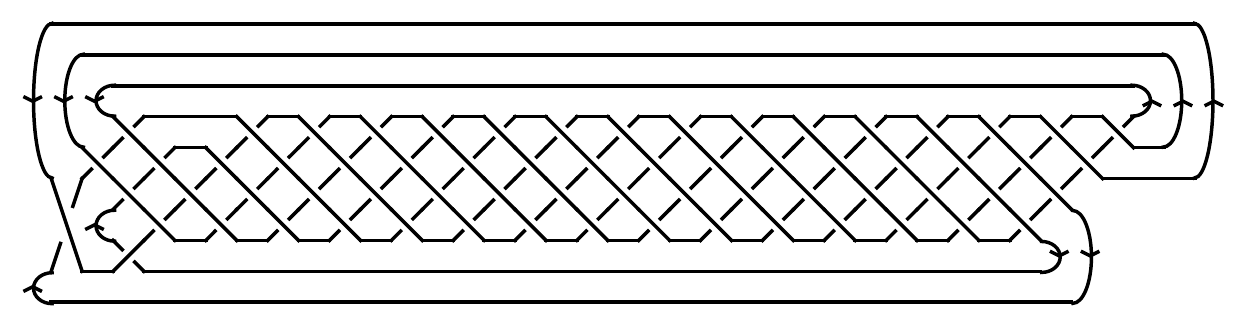}
          \caption*{81}
\end{subfigure}

\begin{subfigure}[b]{0.8\textwidth}
          \centering
          \includegraphics[width=\textwidth]{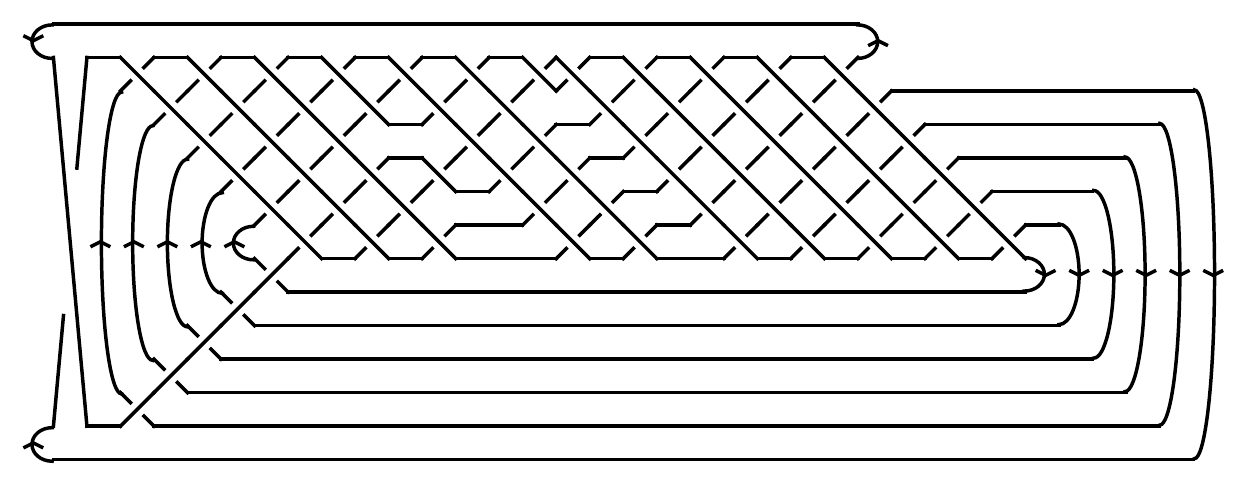}
          \caption*{83}
\end{subfigure}

 \begin{subfigure}[b]{0.8\textwidth}
           \centering
           \includegraphics[width=\textwidth]{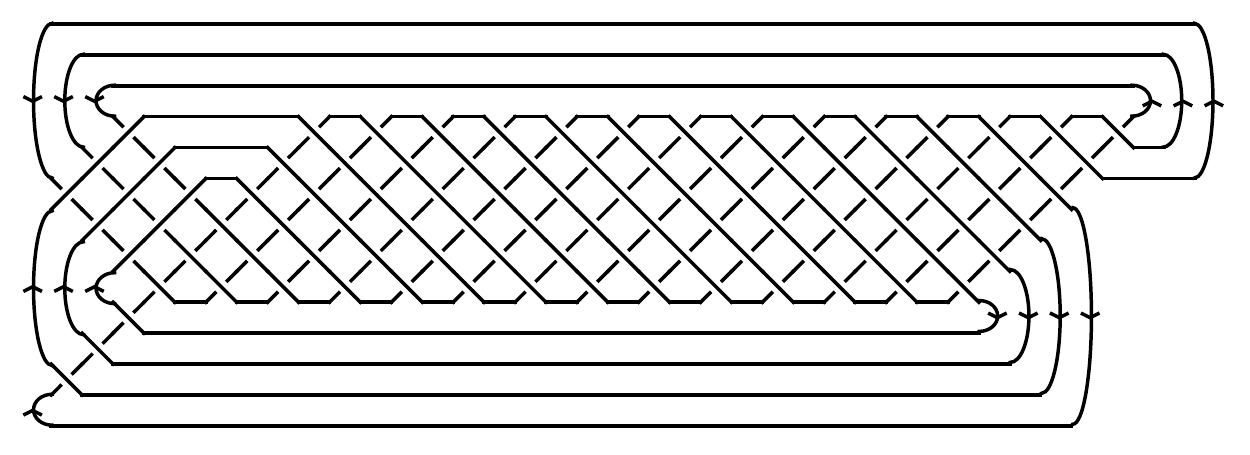}
           \caption*{86}
\end{subfigure}

\begin{subfigure}[b]{0.8\textwidth}
         \centering
         \includegraphics[width=\textwidth]{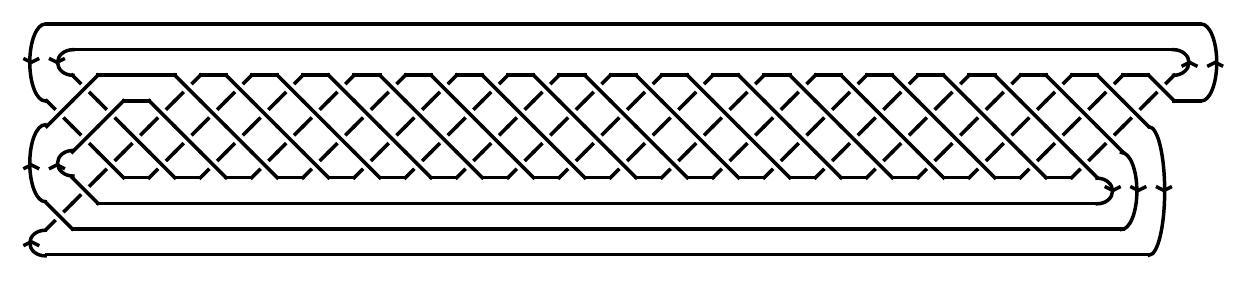}
         \caption*{94}
\end{subfigure}

\begin{subfigure}[b]{0.8\textwidth}
	\centering
           \includegraphics[width=\textwidth]{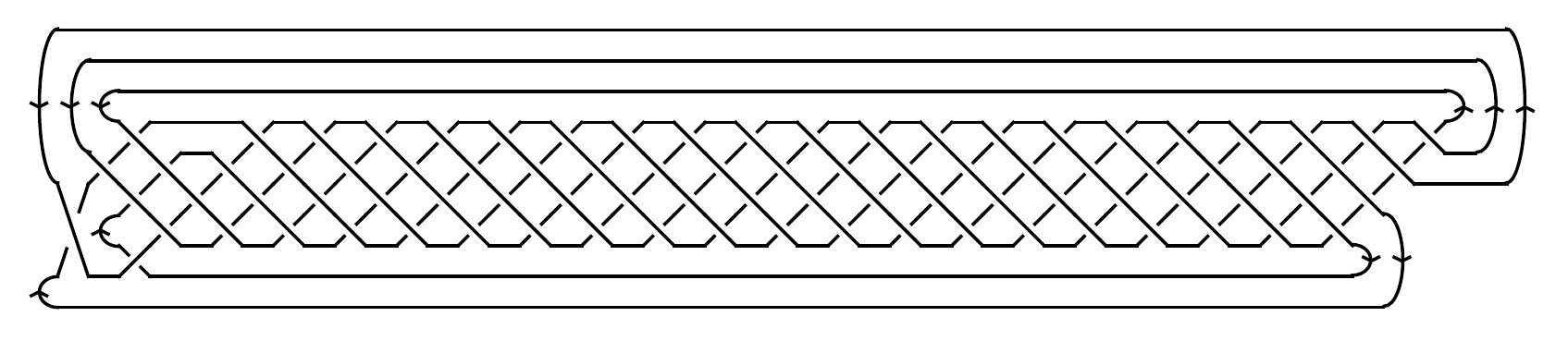}
           \caption*{106}
\end{subfigure}

\begin{subfigure}[b]{0.8\textwidth}
	\centering
           \includegraphics[width=\textwidth]{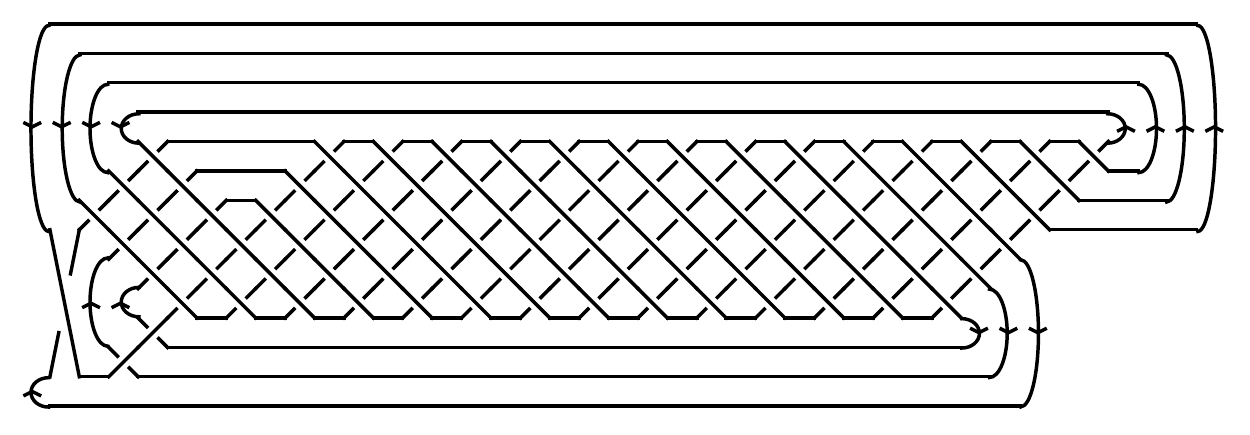}
           \caption*{110}
\end{subfigure}

\end{figure}

\begin{figure}[h]
	\centering
      
\begin{subfigure}[b]{0.8\textwidth}
	\centering
           \includegraphics[width=\textwidth]{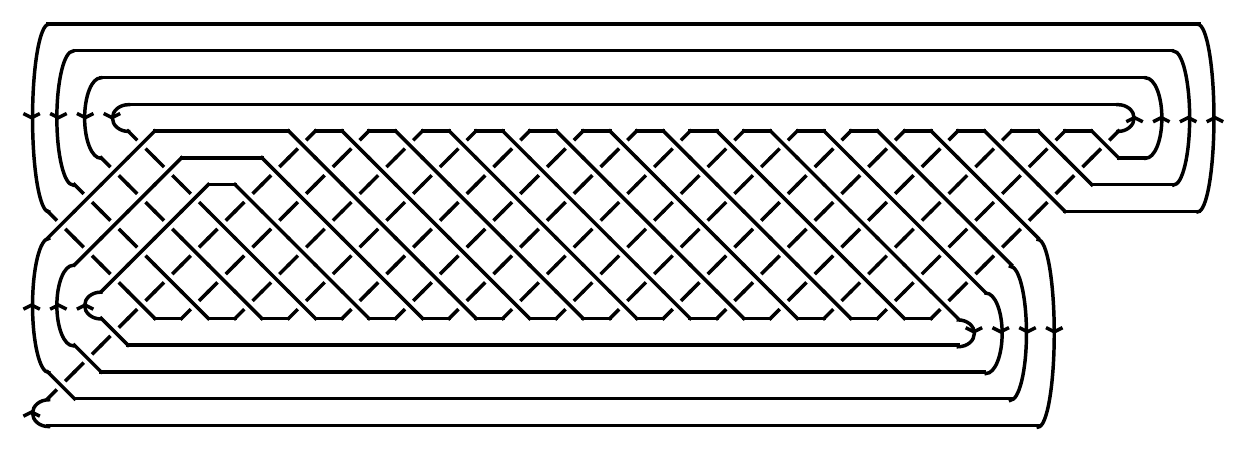}
           \caption*{113}
\end{subfigure}
  
\begin{subfigure}[b]{0.8\textwidth}
           \centering
           \includegraphics[width=\textwidth]{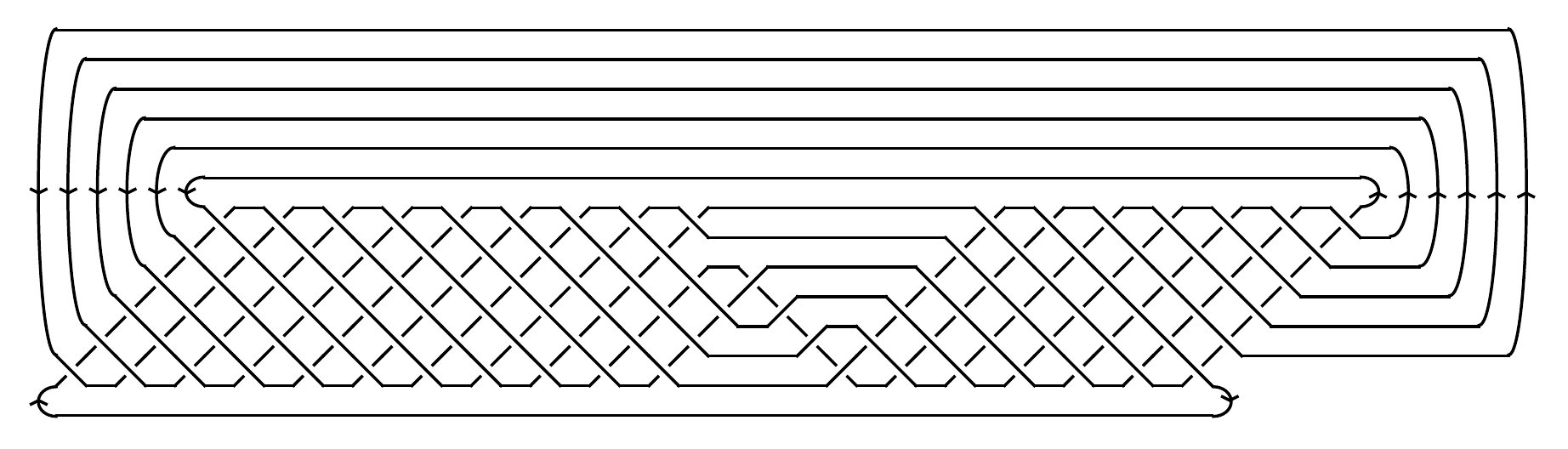}
           \caption*{113}
\end{subfigure}

 \begin{subfigure}[b]{0.8\textwidth}
          \centering
          \includegraphics[width=\textwidth]{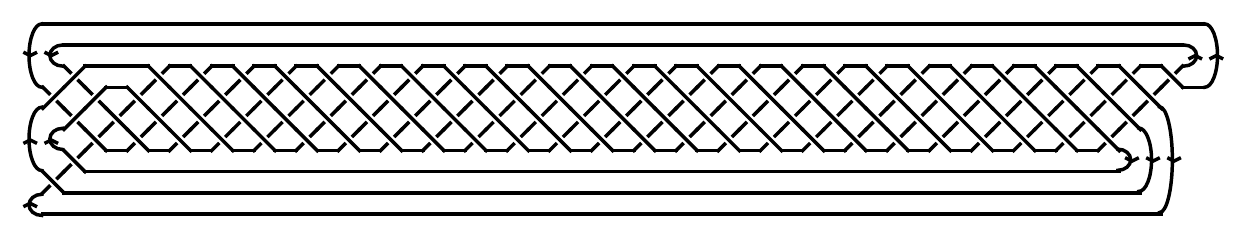}
          \caption*{119}
\end{subfigure}

 \begin{subfigure}[b]{0.8\textwidth}
         \centering
         \includegraphics[width=\textwidth]{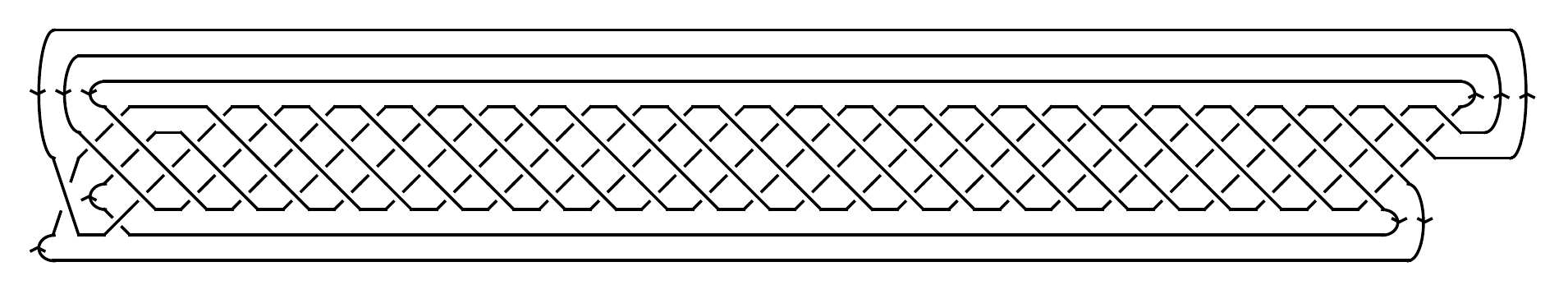}
         \caption*{131}
\end{subfigure}

   \begin{subfigure}[b]{0.8\textwidth}
          \centering
          \includegraphics[width=\textwidth]{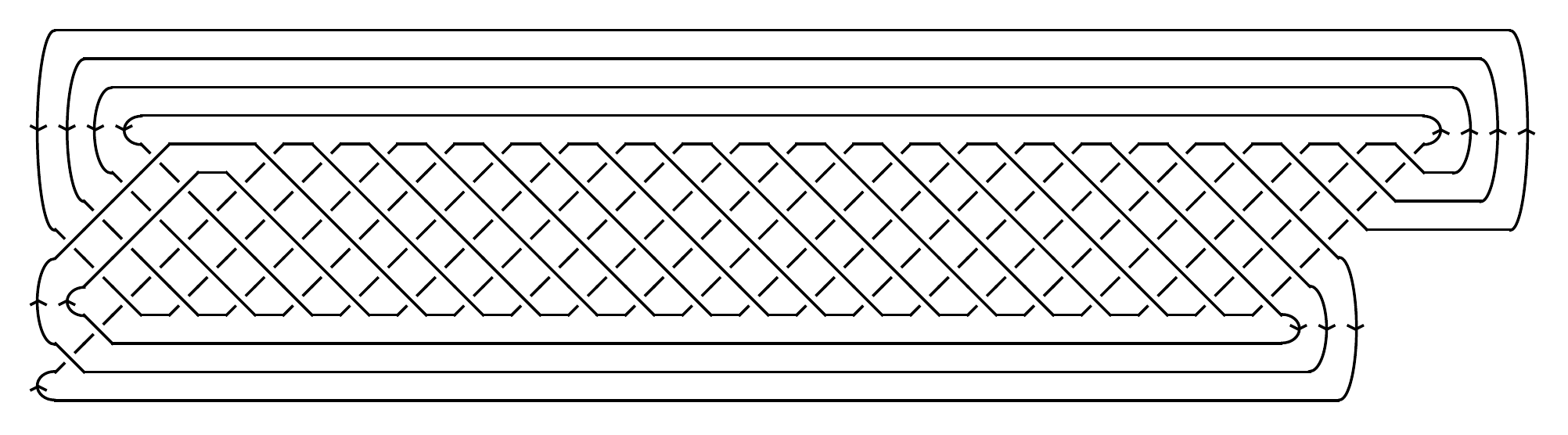}
          \caption*{137}
\end{subfigure}
  
\begin{subfigure}[b]{0.8\textwidth}
          \centering
          \includegraphics[width=\textwidth]{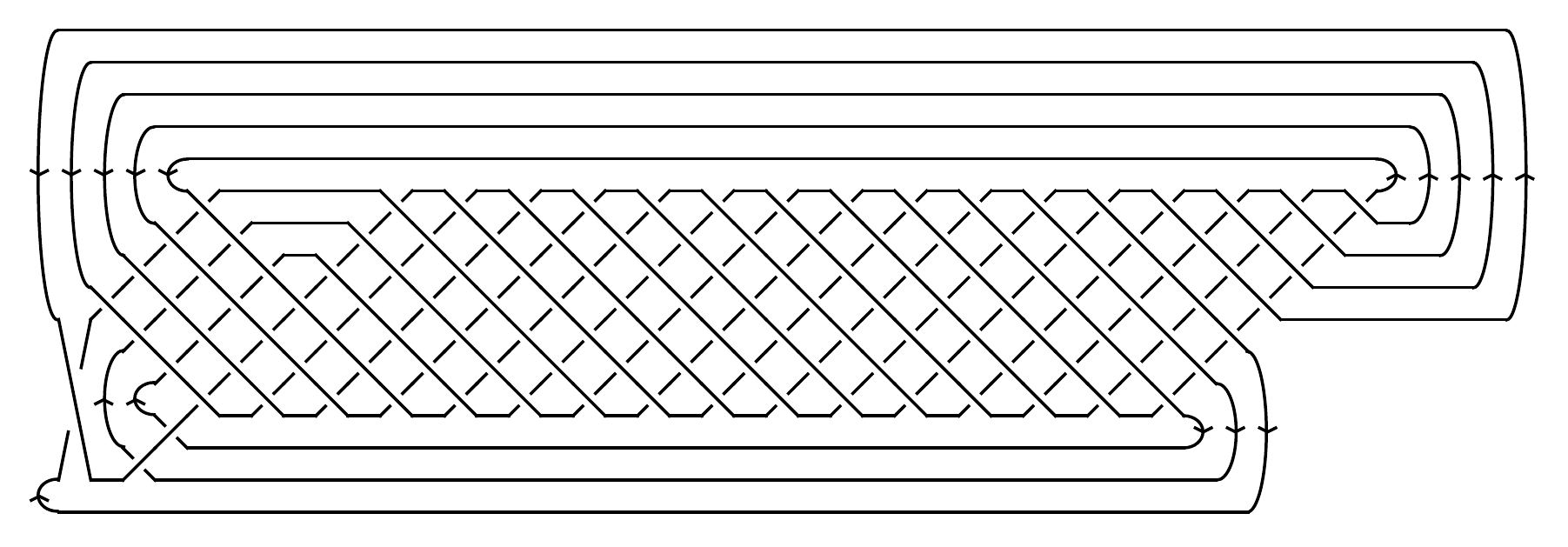}
          \caption*{143}
\end{subfigure}

 \begin{subfigure}[b]{0.8\textwidth}
         \centering
         \includegraphics[width=\textwidth]{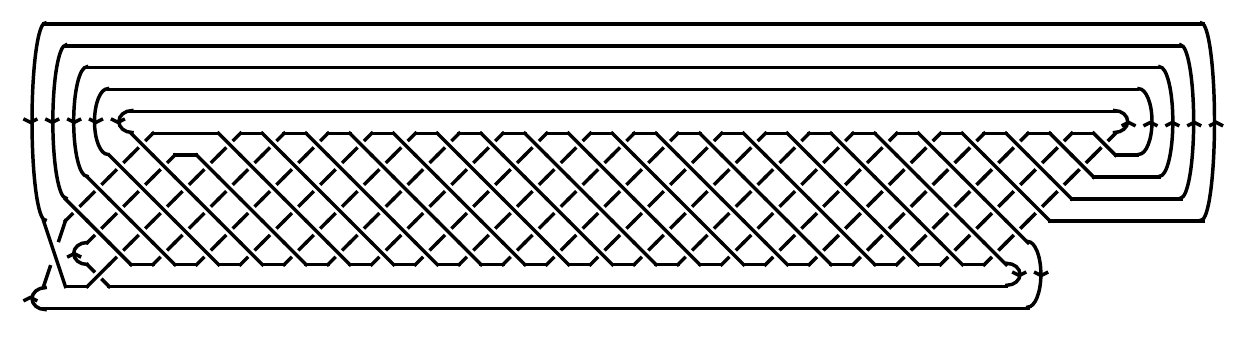}
         \caption*{157}
\end{subfigure} 
 
\end{figure}
\FloatBarrier

\end{document}